\documentclass{amsart}
\RequirePackage[utf8]{inputenc}
\usepackage{amsmath,amsthm,amsfonts,amssymb,amscd,amsbsy,multirow,tikz,hyperref}
\usepackage[all]{xy}
\usepackage{graphicx,pgf,caption,subcaption}
\usepackage{enumerate}
\usepackage{mathdots}
\usepackage{dsfont}
\usepackage{cleveref}
\usepackage{stmaryrd}

\usepackage{tikz, pgfplots}
\usetikzlibrary{positioning}

\usepackage[colorinlistoftodos,disable]{todonotes}

\newcommand{\sphere}{\mathbb{S}}
\newcommand{\DD}{\mathbb{D}}

\newcommand{\F}{\mathcal{F}}

\newcommand{\B}{\mathcal{B}}

\newcommand{\Av}{\operatorname{Av}}
\newcommand{\Vol}{\operatorname{Vol}}
\newcommand{\R}{\mathds R}
\newcommand{\C}{\mathds C}
\newcommand{\N}{\mathds N}

\newcommand{\codim}{\operatorname{codim}}

\newcommand{\diam}{\operatorname{diam}}
\newcommand{\vol}{\operatorname{vol}}
\newcommand{\Hr}{\mathds H}
\newcommand{\Ca}{\mathds{C}\mathrm{a}}
\renewcommand{\O}{\mathsf O}

\newcommand{\lam}{\lambda}

\allowdisplaybreaks

\hypersetup{
    pdftoolbar=true,
    pdfmenubar=true,
    pdffitwindow=false,
    pdfstartview={FitH},
    pdftitle={Basic spectrum},
    pdfauthor={Samuel Lin, Ricardo A. E. Mendes, Marco Radeschi},
    pdfsubject={},
    pdfkeywords={}
    pdfnewwindow=true,
    colorlinks=true, 
    linkcolor=blue,
    citecolor=blue,
    urlcolor=black,
}

\newtheorem{theorem}{Theorem}[]
\newtheorem{lemma}[theorem]{Lemma}
\newtheorem{proposition}[theorem]{Proposition}

\newtheorem{mainthm}{\sc Theorem}

\theoremstyle{definition}
\newtheorem{definition}[theorem]{Definition}

\theoremstyle{remark}
\newtheorem{remark}[theorem]{Remark}

\newtheorem{example}[theorem]{Example}

\title[A Weyl's Law for Singular Riemannian Foliations]{A Weyl's Law for Singular Riemannian Foliations with Applications to Invariant Theory}
\author[S. Lin]{Samuel Lin}
\address{University of Oklahoma\newline
\indent Department of Mathematics\newline
\indent 601 Elm Ave\newline
\indent Norman, OK, 73019-3103, USA}
\email{Samuel.Z.Lin-1@ou.edu}

\author[R. A. E. Mendes]{Ricardo A. E. Mendes}
\address{University of Oklahoma\newline
\indent Department of Mathematics\newline
\indent 601 Elm Ave\newline
\indent Norman, OK, 73019-3103, USA}
\email{ricardo.mendes@ou.edu}

\author[M. Radeschi]{Marco Radeschi}
\address{Universit\`a degli Studi di Torino\newline
\indent Departimento di Matematica ``G. Peano''\newline
\indent Via Carlo Alberto, 10\newline
\indent 10123 Torino (TO), Italy}
\email{marco.radeschi@unito.it}

 \thanks{The first author was partially supported by the Bridge Funding Investment Program at the University of Oklahoma. The second author was partially supported by NSF grant DMS-2005373 and the Dodge Family College of Arts and Sciences Junior Faculty Summer Fellowship of the University of Oklahoma. The third author was partially supported by NSF grant DMS-1810913 and NSF CAREER grant DMS-2042303}

\begin{document}

\begin{abstract}
We prove a version of Weyl's Law for the basic spectrum of a closed singular Riemannian foliation $(M,\F)$ with basic mean curvature. In the special case of $M=\sphere^n$, this gives an explicit formula for the volume of the leaf space $\sphere^n/\F$ in terms of the algebra of basic polynomials. In particular, $\Vol(\sphere^n/\F)$ is a rational multiple of $\Vol(\sphere^m)$, where $m=\dim (\sphere^n/\F)$.
\end{abstract}

\maketitle


\section{Introduction}

\todo{Abstract and introduction added}
Given a compact, connected $n$-dimensional Riemannian manifold $M$, the Laplacian operator $\Delta$ is a self-adjoint operator with a discrete spectrum $0=\tilde\lambda_0<\tilde\lambda_1\leq \tilde\lambda_2\leq \ldots$. While the eigenvalues themselves strongly depend on the geometry of $M$, the celebrated Weyl's theorem states that the \emph{growth} of the eigenvalues only depends on the dimension and volume of $M$: in other words, letting $S(t)=\max\{k\mid \tilde\lambda_k< t\}$ denote the counting function for the eigenvalues of $\Delta$, we have that
\[
S(t)\sim {\omega_n\over (2\pi)^n}\Vol(M)t^{n/2},
\]
where $\omega_n$ denotes the volume of the unit $n$-ball, and $f\sim g$ means that $\lim_{t\to \infty}{f(t)\over g(t)}=1$. Weyl's Law was generalized by Br\"uning-Heintze \cite{BH78} and Donnelly \cite{Donnelly78} to the context of compact manifolds equipped with an isometric action by a compact Lie group $G$; in their version, the growth of the counting function for the \emph{$G$-invariant} eigenfunctions of $\Delta$ depends on the volume and dimension of the \emph{orbit space} $M/G$.
\\

\emph{Singular Riemannian foliations} $(M,\F)$ are, roughly speaking, partitions of a Riemannian manifold $M$ into connected submanifolds (called \emph{leaves}) locally equidistant to one another (cf. Section \ref{SS:defs} for a more detailed definition). They generalize the decomposition into orbits of isometric actions of connected groups, which we also refer to as \emph{homogeneous} singular Riemannian foliations. Singular Riemannian foliations also appear as natural generalizations of isoparametric foliations in symmetric spaces, Riemannian submersions, and the dual foliation in non-compact manifolds with non-negative curvature. Recently, the understanding of singular Riemannian foliations has seen rapid development, and despite providing a proper generalization of group actions (that is, there exist infinitely many non-homogeneous examples, even in spheres, cf. \cite{Radeschi14}), it has been shown that they retain a lot of the rigid properties that isometric actions enjoy \cite{GGR15,GR15,MR19}. This is especially true when $M$ is a round sphere $\sphere^n$, which appears as an ``infinitesimal model'' for general singular Riemannian foliations around a point \cite{LR18, MR20q}.

Given a singular Riemannian foliation $(M,\F)$ with closed leaves, one can define a \emph{leaf space} $M/\F$ which is an analogue to the orbit space, and the algebra of \emph{basic functions} defined as the functions which are constants along the leaves. If $\F$ has \emph{basic mean curvature}, in the sense that the mean curvature vectors of the leaves of $\F$ project to a vector field in $M/\F$, then the Laplace operator restricts to a self-adjoint operator on the space of basic functions, with a discrete spectrum $0=\lambda_0<\lambda_1\leq \lambda_2\leq \ldots$ (cf. Section \ref{SS:princ}). The first main result can then be described as Weyl's Law for closed singular Riemannian foliations:

\begin{mainthm}
\label{MT:Weyl}
Let $M$ be a compact Riemannian manifold and $\F$ a singular Riemannian foliation of $M$ with closed leaves and basic mean curvature. Let $N(t)=\max\{i\mid \lambda_i< t\}$ be the counting function for the basic spectrum, and $X=M/\F$ be the leaf space. Then 
\[N(t)\sim {vol(X)\omega_m\over (2\pi)^m}t^{m/2}\qquad \text{as}\quad t\to \infty \]
where $m=\dim X$ and $\omega_m=vol \DD^m$.
\end{mainthm}

The condition of basic mean curvature is always satisfied when $\F$ is homogeneous or when $M$ is the Euclidean space, a round sphere \cite{AR15}, as well as complex and quaternionic projective spaces \cite[page 2205]{AR16}.

In the homogeneous case, Theorem \ref{MT:Weyl} follows as a special case of Corollary 3.5 of Br\"uning-Heintze \cite{BH78} and the Main Theorem of Donnelly \cite{Donnelly78}.  In the \emph{regular case}, namely when all leaves have the same dimension, this result was proved by Richardson \cite{Richardson10} without the assumptions of closed leaves and basic mean curvature.
\\

A surprising result we prove using Weyl's Law is the following:

\begin{mainthm}\label{MT:volume}
If $(\sphere^n,\F)$ is a closed singular Riemannian foliation  (resp. $G\to \O(n+1)$ is an orthogonal representation) with $m$-dimensional leaf space (resp. orbit space), then $\Vol(\sphere^n/\F)$ (resp. $\Vol(\sphere^n/G)$) is a rational multiple of 
$\Vol(\sphere^m)$.
\end{mainthm}

To the best of the authors' knowledge, Theorem \ref{MT:volume} is novel even in the homogeneous case, even though it is clearly true when $G$ is finite. Since the quotients $\sphere^n/\F$ can be thought as singular spaces all of whose geodesics are closed (in the sense of quotient geodesics from \cite[Section 3.4]{MR20}), Theorem \ref{MT:volume} is similar in spirit to \cite[Theorem 3]{Wei77} which applies to \emph{manifolds} all of whose geodesic are closed.

In the case of foliations, Theorem \ref{MT:volume} follows immediately from Theorem \ref{MT:spherical} below, in which the rational number $\Vol(\sphere^n/\F)/\Vol(\sphere^m)$ is computed explicitly in terms of the \emph{algebraic structure} of $(\sphere^n,\F)$. Recall that a singular Riemannian foliation $(\sphere^n,\F)$ extends to an \emph{infinitesimal foliation} $(\R^{n+1},\F)$, i.e., a singular Riemannian foliation with the origin as a leaf. In this case, it was shown in \cite{MR20} that the algebra $A$ of basic polynomials is a so-called \emph{Laplacian algebra} (i.e., a graded algebra preserved by the multiplication by $r^2=\sum_i x_i^2$ and by the Laplacian $\Delta=\sum_i{\partial^2\over \partial x_i^2}$) which completely determines $\F$. Furthermore, there is a dictionary between the geometric properties of an infinitesimal foliation and the algebraic properties of its Laplacian algebra $A$. The next main theorem adds a line to this dictionary:

\begin{mainthm}
\label{MT:spherical}
Let $(V,\F)$, $\dim V=n+1$, be an infinitesimal foliation with closed leaves, let $A=\bigoplus_i A_i\subset\R[V]$ be the associated Laplacian algebra of basic polynomials, and let $H(z)=\sum_i (\dim A_i) z^i$ denote its Hilbert series. Denote by $\sphere V$ the unit sphere in $V$ and by $m$ the dimension of the leaf space $X=\sphere V /\F$.

Then:
\begin{enumerate}[(a)]
\item The Laurent series of $H(z)$ at $z=1$ has the form
\[H(z)= \frac{\Vol(X)}{\Vol(\sphere^m)} (1-z)^{-m-1} + (\cdots) (1-z)^{-m} + \cdots\]
\item $A$ is a \emph{free module over a free subalgebra}, i.e. $A=\bigoplus_{i=1}^\ell \R[f_0,\ldots, f_m]g_j$ for some homogeneous $f_0,\ldots f_m, g_1,\ldots g_\ell\in A$, and
\[\frac{\Vol(X)}{\Vol(\sphere^m)}= \frac{\ell }{\deg f_0 \cdots \deg f_m}.\] 
\end{enumerate}
\end{mainthm}

The analogue of Theorem \ref{MT:spherical} for representations of disconnected groups also holds; see Proposition \ref{P:spherical}.
\\

Let us comment on the main aspects of the proofs. Theorem \ref{MT:Weyl} follows a framework similar to the homogeneous version by Br\"uning-Heintze \cite{BH78}. In their proof of Weyl's Law, however, results from transformation groups and representation theory were used in a fundamental way in many places, which are not available in the foliated case. Some of these obstacles have been overcome using recent results in the literature of singular Riemannian foliations, most notably the averaging operator introduced in \cite{LR18} and the foliated slice theorem in \cite{MR19}. Others require new uniform estimates on the geometry of leaves (see Section \ref{S:geometricestimate}), which culminate in the following result:
\begin{mainthm}
\label{MT:geometricestimate} 
Given a compact connected Riemannian manifold $M$ and a singular Riemannian foliation $\F$ with closed leaves, there exists a constant $C=C(M,\F)$ such that, for every $p\in M$, and every $s>0$, one has 
\[
\int_{L_p}\exp\left(-{d_M(p,q)^2\over s}\right)dq\leq Cs^{\dim L_p\over 2}.
\]
where $L_p$ denotes the leaf through $p$.
\end{mainthm}

The proof of Theorem \ref{MT:spherical} relies on the observation that the Hilbert series $H(z)$ for the Laplacian algebra $A$ of basic polynomials of $(V, \F)$ contains, via the theory of spherical harmonics, the same information as the basic spectrum for the Laplacian of $\sphere V$. Using well-known techniques from enumerative combinatorics, we relate the pole of $H(z)$ at $z=1$ with the growth of the coefficients of $H(z)$, hence to the growth of basic eigenvalues of $(\sphere V, \F)$, which is given by the foliated Weyl's Law (Theorem \ref{MT:Weyl}).

The paper is structured as follows: in Section \ref{S:prelim}, we recall the main definitions and important results that we will use throughout the paper. Section \ref{S:geometricestimate} has a more geometric flavor and is devoted to the proof of Theorem \ref{MT:geometricestimate}.  In Section \ref{S:Weyl}, which has a more analytic flavor, we prove Theorem \ref{MT:Weyl}. Finally, the more algebraic Section \ref{S:spherical} contains the proof of Theorem \ref{MT:spherical}.

\subsection*{Acknowledgements} It is a pleasure to thank Emilio Lauret for pointing out the paper of Donnelly \cite{Donnelly78}, which kickstarted the project. We would also like to thank Ken Richardson for insightful comments and for pointing out further existing literature on the subject. Finally, we would like to thank Alexander Lytchak for many conversations regarding the interaction between the results in the current manuscript and his recent preprint \cite{Lytchak23}. This paper was written in part during the second author's visits to the third author at the University of Notre Dame and the University of Torino. The second author thanks both universities for their hospitality.

%
%
\section{Preliminaries}\label{S:prelim}
\subsection{Singular Riemannian foliations}

\subsubsection{Basic definitions and facts.} \label{SS:defs} The main objects of study in the present paper are \emph{singular Riemannian foliations}, first defined in \cite{Molino}, see also \cite{RadLN}. By definition, a singular Riemannian foliation is a partition $\F$ of a Riemannian manifold $M$ into connected, injectively immersed smooth submanifolds
(called \emph{leaves}), which is generated by a set of smooth vertical fields, and such that every geodesic of $M$ that meets one leaf orthogonally is orthogonal to every other leaf that it meets. The last condition is equivalent to the leaves being locally at a constant distance apart \cite[Proposition 2.3]{RadLN}.

The main examples are fibers of a Riemannian submersion, leaves of a (regular) Riemannian foliation, the parallel and focal sets of an isoparametric submanifold of a sphere, and orbits of an isometric group action (sometimes called ``homogeneous'' singular Riemannian foliations).

We say $\F$ is \emph{closed} if all the leaves are closed. In this case, the condition involving normal geodesics is equivalent to (global) equidistance of every pair of leaves \cite[Remark 2.5]{RadLN}.

We say $\F$ is \emph{infinitesimal} when $M$ is isometric to some Euclidean space $\R^{n+1}$, and the origin $\{0\}$ is a leaf. In this case any sphere centered at the origin is a union of leaves (in other words, is ``saturated'' by $\F$), and, by Molino's Homothetic Transformation Lemma \cite[Lemma 6.2]{Molino}, $\F$ is determined by its restriction to any such sphere. In fact, the map $\F\mapsto \F|_{\sphere^{n}}$ is a bijection between infinitesimal singular Riemannian foliations of $\R^{n+1}$ and singular Riemannian foliations of $\sphere^{n}$.

Given a singular Riemannian foliation $(M,\F)$, the set of leaves is called the \emph{leaf space}, and is denoted by $M/\F$. When the leaves are closed, the leaf space $M/\F$ has a natural metric structure, with respect to which the natural projection map $\sigma\colon M\to M/\F$ is a ``manifold submetry''. See \cite{KL20, MR20} and references therein for more information on submetries and manifold submetries.

A function defined on $M$ is \emph{ $\F$-basic} (or just \emph{basic}, if $\F$ is understood from context) when it is constant on the leaves of $\F$. Equivalently, when it is the pull-back by $\sigma$ of a function defined on the ``base'' $M/\F$. We sometimes denote sets of basic functions with the superscript $^b$, for example, $C^\infty(M)^b$ for the set of all smooth basic (real-valued) functions on $M$.

\subsubsection{Invariant Theory.} \label{SS:InvT} If $(V,\F)$ is an infinitesimal closed singular Riemannian foliation, the algebra of all basic \emph{polynomials}, denoted $A=\B(\F)=\R[V]^b$, is large enough that it separates leaves, and moreover, it is a finitely generated algebra. This is known as the Algebraicity Theorem \cite{LR18}. When $(V,\F)$ is homogeneous, that is, when the leaves are given as the orbits of an orthogonal representation of a compact group, this algebra $A$ is called the ``algebra of invariants'', and the Algebraicity Theorem reduces to classical facts in Invariant Theory.

\subsubsection{Principal part.} \label{SS:princ} Given a singular Riemannian foliation $(M,\F)$, there exists an open, dense, saturated, full-measure subset $M_0\subset M$, called the \emph{principal part}, such that $M_0/\F \subset M/\F$ is a smooth Riemannian manifold, and $\sigma|_{M_0}\colon M_0\to M_0/\F$ is a Riemannian submersion (see \cite[page 3]{LR18}). (To justify the name, note that, when $\F$ is homogeneous, $M_0$ is the union of the principal orbits.) Note that, when $M$ is compact and $\F$ is closed, the map $\sigma|_{M_0}\colon M_0\to M_0/\F$ is proper, hence a locally trivial fiber bundle by Ehresmann's Lemma. In particular, the function $h\colon M_0/\F\to\R$, given by $h(x)=\vol (\sigma^{-1}(x))$, and its pull-back $h\circ \sigma|_{M_0}\colon M_0\to\R$, given by $p\mapsto \vol(L_p)$, are smooth.

We denote by $H(p)\in T_pM$ the mean curvature vector of the leaf $L_p$ as a submanifold of $M$, for every $p\in M_0$. Note that $H$ is a smooth vector field on $M_0$. We say $\F$ has \emph{basic mean curvature} when $H$  is $\sigma$-related to a (smooth) vector field $\underline{H}$ on $M_0/\F$. By \cite[Prop. 3.1]{AR15}, this is automatically the case when $M$ is a round sphere, or $M$ is Euclidean space and $\F$ is infinitesimal.

\subsubsection{Tubular neighborhoods and the Slice Theorem.} \label{SS:slice} Let $M$ be a closed Riemannian manifold, and $\F$ a closed singular Riemannian foliation. Fix a leaf $L$, and $\epsilon>0$ be smaller than the normal injectivity radius of $L$. Denote by $\nu L$ the normal bundle of $L$, by $\nu L^{\leq\epsilon}$ (respectively  $\nu L^{<\epsilon}$) the set of all vectors in $\nu L$ of length at most $\epsilon$ (resp. less than),  by $B_\epsilon(L)$ the (open) tube of radius $\epsilon$ around $L$, that is, the set of all points at distance less than $\epsilon$ from $L$, and by $\exp\colon \nu L \to M$ the normal exponential map. Note that the restriction of $\exp$ to $\nu L^{<\epsilon}$  is a diffeomorphism onto $B_\epsilon(L)$. We denote by $\pi\colon B_\epsilon(L)\to L$ the nearest-point projection, and note that $\pi\circ\exp$ coincides with the natural projection $\nu L\to L$ on $\nu L^{<\epsilon}$. Note also that the restriction of $\pi$ to any leaf $L'\subset B_\epsilon(L)$ is a submersion $L'\to L$. This follows from \cite[Theorem 2.2(b)]{ABT13}, see also \cite[Proposition 13]{MR19}).

Let $p\in L$. The \emph{disconnected slice foliation} (see \cite[Section 3.3]{MR19}) of $\F$ at $p$, denoted by $\F^p$, is a partition of $\nu_p L$ obtained by first intersecting the leaves of $\F$ with $\exp \nu _p L^{\leq\epsilon} $, then pulling back this decomposition by $\exp$, and finally
 extending it via homothetic transformations to all of $\nu_p L$. This partition $\F^p$ of the Euclidean space $\nu_p L$  is a ``disconnected infinitesimal foliation'' in the sense of  \cite[Definition 8]{MR19}, and its restriction to any sphere centered at the origin is a ``disconnected singular Riemannian foliation'' in the sense of Definition \ref{D:disco} below. By the Slice Theorem \cite[Theorem A]{MR19}, the disconnected slice foliation $\F^q$ at any other $q\in L$ is equivalent to $\F^p$ in the sense that there exists an isometry $\nu_p L\to\nu_q L$ sending leaves onto leaves.

\subsubsection{Averaging operator.} \label{SS:av} Given a singular Riemannian foliation $(M,\F)$ with compact leaves, and a $L^2$
function $f\colon M\to \R$, we define its average $\Av(f)\colon M\to \R$ over $\F$ by $\Av(f)(p)={1\over \vol(L_p)}\int_{q\in L_p} f(q) d\vol^{L_p}(q)$, where $d\vol^{L_p}$ denotes the Riemannian volume form of the compact Riemannian manifold $L_p$ and $\vol(L_p)$ its volume. We will frequently write $dq$ instead of $ d\vol^{L_p}(q)$ to simplify notation. When $\F$ has basic mean curvature, the averaging operator $\Av$ takes smooth functions to smooth functions, and commutes with the Laplace operator on the principal part $M_0$, see \cite[Thm 3.3 and Lemma 3.2]{LR18}, see also \cite{PR96} for similar considerations in the case of  regular Riemannian foliations.

\subsection{H\"ormander's estimate}


We briefly recall H\"ormander's estimate on spectral functions, which is one of the main analytical ingredients in the proof of Weyl's law for singular Riemannian foliations.

Let $\mathcal{H}$ be a Hilbert space with inner product $( , )_{\mathcal{H}}$ and let $A$ be a self-adjoint, possibly unbounded linear operator on $\mathcal{H}$.
Given any characteristic function $\chi_{\mathcal{U}}$ on a Borel measurable set $\mathcal{U}$ in $\mathbb{R}$, one can use Borel functional calculus to define an orthogonal projection $\chi_{\mathcal{U}}(A)$ on $\mathcal{H}$  \cite[p. 262]{ReedSimon}. 

Given any $\lam\in \mathbb{R}$, define $E(\lam):=\chi_{(-\infty, \lam)} (A)$.
The family of projections $\{E(\lam)\}$ is called the \textit{spectral resolution} of the self-adjoint linear operator $A$.

\begin{example}
Suppose that $A\colon  \mathcal{H}\to \mathcal{H}$ is a positive, self-adjoint linear operator with discrete spectrum
\[
0< \lam_1 \leq \lam_2 \leq \lam_3 \cdots  
\]
and an orthonormal eigenbasis basis $\{u_i\}$. By the spectral theorem in projection-valued measure form \cite[Theorem VIII.6]{ReedSimon}, the spectral resolution of $A$ is simply the projection to the sum of eigenspaces whose corresponding eigenvalues are less than $\lam$. In other words, 
\begin{equation} \label{eqn:DiscreteResolution}
E(\lam)(f)=\chi_{(-\infty, \lam)} (A)(f)=\sum_{\lam_i <\lam} (f, u_i)_{\mathcal{H}} \cdot u_i.
\end{equation}
\end{example}

To define spectral functions, let $\Omega$ be a smooth, $m$-dimensional Riemannian manifold with volume measure $\mu$.
Let $\hat{P}\colon  L^2(\Omega, \mu)\to L^2(\Omega, \mu)$ be a formally positive self-adjoint extension of an elliptic operator $P$ on $C^{\infty}(\Omega)$. Let $\{E(\lam)\}$ be the spectral resolution of $\hat{P}$. It is a standard result \cite[Section 3.7]{BF02392492} that the ellipticity of $P$ implies that for each $t$, there exists a unique kernel $\theta(x, y;t) \in C^{\infty}(\Omega\times \Omega)$ such that 
for any $f \in L^2(\Omega, \mu)$,
\begin{equation} \label{eqn:SpectralFunction}
E(t) (f)(x)=\int_{\Omega} \theta(x, y; t)f(y) d\mu(y).
\end{equation}
The kernels $\theta(x, y; t)$ are called \textit {spectral functions}.
H\"{o}rmander proved the following asymptotics for the spectral functions:  
\begin{theorem} \label{T:Hormander} \cite{BF02391913}
Let $p$ be the principal symbol of an elliptic differential operator $P$ of order $k$ and let $d\xi$ be the measure induced by $d\mu$ on each fiber of the cotangent bundle $T^{*}\Omega$. Define
\[
R(x, \lam):=\lam^{-m/k}\theta(x,x; \lam)-(2\pi)^{-m}\int_{B_x} d\xi,
\]
where $B_x=\{\xi\in T_{x}^{*}\Omega\,|\, p(\xi)<1\}$. Then on every compact subset of $\Omega$, $R(x, \lam)=O(\lam^{-1/m})$ uniformly as $\lam \to \infty$.
\end{theorem}

\subsection{Basic Laplacian}\label{SS:Basic-Lap}
Let $(M, g)$ be a compact Riemannian manifold, $\F$ a closed singular Riemannian foliation with basic mean curvature, and let $\Delta(f):=-\mathrm{div} \nabla (f)$ be the Laplace operator on functions.
Recall from Section \ref{SS:av} that the averaging operator $\Av$ commutes with the Laplace operator $\Delta$ on $C^{\infty}(M)$. Therefore, the Laplace operator $\Delta$ sends basic smooth functions to basic smooth functions. We then extend $\Delta |_{C^{\infty}(M)}$ to obtain a self-adjoint linear operator $R\colon  L^2(M)^b \to L^2(M)^b$ with domain $\mathcal{D}(R)=H^2(M)^{b}$.

Let $M_0$ be the principal part of $M$ and let $X_0:=\sigma (M_0)$ be the manifold part of $X$ such that $\sigma\colon  M_0\to X_0$ is a Riemannian submersion. Given any point $x\in X_0$, define $h(x)$ to be the volume of the leaf $L_{\sigma^{-1}(x)}$, and recall from Section \ref{SS:princ} that $h$ is a smooth function.
By Fubini's theorem for Riemannian submersions, the pull-back of $\sigma$ is a Hilbert space isometry $\sigma^{*}\colon  L^{2}(X_0, h) \to L^2(M_0)^b$, where $L^{2}(X_0, h)$ is the set of squared-integrable functions on $X_0$ with respect to the measure $h(x)\,d\mathrm{vol}_{X_0}$.

Let $\sigma_{*}$ be the inverse of $\sigma^{*}$. Define $T\colon  L^{2}(X_0, h) \to L^{2}(X_0, h)$ 
as
\[
	T:= \sigma_{*} \circ R\circ \sigma_{*}^{-1}.
\]
A simple calculation proves the following lemma that relates the operator $T$ to the Laplace operator on $X_0$:
\begin{lemma} \label{L:SymbolT}
For any $f\in C^{\infty}(X_0) \cap L^{2}(X_0, h)$, we have 
\[
	T(f)=\Delta_{X_0} f+ \underline H \, f, 
\]
where $\underline H$ is the projection of the mean curvature vector field $H$ through $\sigma$.
\end{lemma}
By Lemma \ref{L:SymbolT}, the operator $T$ is a self-adjoint extension of an elliptic differential operator, whose principal symbol is the same as that of the Laplace operator on $X_0$. 

\begin{remark} \label{R:Eigenbasis}
As $M$ is compact, the Laplace operator has discrete spectrum:
\[
0=\tilde{\lam}_0<\tilde{\lam}_1\leq \tilde{\lam}_2 \cdots \leq \tilde{\lam}_k\leq\cdots.
\]
Since $T\colon  L^{2}(X_0, h) \to L^{2}(X_0, h)$ is equivalent to $R\colon  L^2(M)^b \to L^2(M)^b$, it must have a discrete spectrum:
\[
0=\lam_0<\lam_1\leq \lam_2 \cdots \leq \lam_i\leq\cdots,
\]
and we can define the counting function
\[
N(t)=\max\{k\mid \lambda_k<t\}.
\]
Furthermore, since the Laplace operator commutes with $\Av$, there exists a smooth, orthonormal basis of eigenfunctions $\{\phi_k\}$ for $L^2(M)$ with a subsequence $\{\phi_{k_i}\}$ that generates the vector subspace $L^2(M)^{b}$.

Define $\varphi_i:=\sigma_{*}(\phi_{k_i})$. Then $\{\varphi_i\}$ form a smooth orthonormal basis of $T$-eigenfunctions for $L^2(X_0, h)$.
\end{remark}

\begin{proposition} \label{P:SpectralFunctionExp}
Let $\theta^{T}(x, y; t)$ and $\theta^{\Delta}(x, y; t)$ be the spectral function of $T$ and $\Delta$ respectively and let $\{\phi_k\}$ and $\{\varphi_i\}$ be the smooth orthonormal bases in Remark \ref{R:Eigenbasis}. Then for each $t\in \mathbb{R}$, $p, q\in M$ and $x, y \in X_0$ we have the following:
\begin{enumerate}[(a)]
\item
$\theta^{T}(x, y; t)=\sum_{\lam_i<t} \varphi_i(x) \varphi_i(y)$ and
$\theta^{\Delta}(p, q; t)=\sum_{\tilde{\lam}_k<t} \phi_k(p) \phi_k(q)$.

\item For each pair of $p, q \in M$ and each $t\in \mathbb{R}$,
\[
\theta^{T}(\sigma(p), \sigma(q) ; t)= (\Av \otimes \Av) \theta^{\Delta} (p, q; t).
\]
\end{enumerate}
\end{proposition}

\begin{proof}
By \eqref {eqn:DiscreteResolution}, $\sum_{\lam_i<t} \varphi_i(x) \varphi_i(y)$ and $\sum_{\tilde{\lam}_k<t} \phi_k(p) \phi_k(q)$ are two smooth functions satisfying \eqref{eqn:SpectralFunction}. Hence (a) follows from the uniqueness of spectral functions.

To prove (b), simply note that
\begin{equation*}
\begin{split}
\theta^{T}(\sigma(p), \sigma(q) ; t)&=\sum_{\lam_i<t} \varphi_i (\sigma(p)) \varphi(\sigma(q))\\
&=\sum_{\lam_i<t} \sigma^{*}\varphi_i (p) \sigma^{*}\varphi_i (q)\\
&=\sum_{\tilde{\lam_k}<t} \Av\phi_k(p) \Av\phi_k(q)= (\Av \otimes \Av) \theta^{\Delta} (p, q; t).
\end{split}
\end{equation*}
\end{proof}

By applying Theorem \ref{T:Hormander} to the spectral function of $T$, we have the following:
\begin{proposition} \label{P:HormanderT}
Let $\theta^{T}(x, y; t)$ be the spectral function of the elliptic differential operator that generates $T$. Then for each $x\in X_0$, 
\[
\theta^{T}(x, x; t) \sim t^{m/2}\frac{\omega_{m}}{h(x)(2\pi)^m},
\]
as $t\to\infty$, where $\omega_m$ is the volume of unit balls in $\mathbb{R}^m$.
\end{proposition}

\subsection{Heat Kernels}\label{SS:heat-K}
Let $\phi_i(x)$ be the orthonormal eigenbasis for $\Delta$ constructed in Remark \ref{R:Eigenbasis}. A well-known result by Minakshisundaram and Pleijel \cite{MR31145} (see also \cite[Theorem 7.2.18]{MR2742784}) states that, for each $\sigma>\dim(M)$, the Dirichlet series $\sum_i  \lambda_{i}^{-\sigma} \phi_i(p)\phi_i(q)$ converges uniformly and absolutely on $M\times M$. Combining with (a) in Proposition \ref{P:SpectralFunctionExp}, we see that the spectral function $\theta^{\Delta}(p, q; t)$ has uniform polynomial growth, in the sense that there exists positive constants $C$ and $N$ such that for all $p, q\in M$,
\[
	|\theta^{\Delta}(p, q; t)|<Ct^N.
\]
By using (a) and (b) in Proposition \ref{P:SpectralFunctionExp}, we see that $\theta^{T}(x, y; t)$ also has uniform polynomial growth.
\todo{Sam: I rephrased this paragraph to make the notations consistent.}
As a consequence, for any $s>0$, $x, y \in X_0$, and $p, q\in M_0$, one can take the Laplace transform of $\theta^{\Delta}$ and $\theta^{T}$ to define
\begin{equation} \label{eqn:HeatKernelDef}
	\begin{split}
	\Gamma_{s}^{T}(x, y)&:=s\int_{0}^{\infty} e^{-st} \theta^{T} (x, y; t) \, dt, \, \text{ and }\\
	\Gamma_{s}^{\Delta}(p, q)&:=s\int_{0}^{\infty} e^{-st} \theta^{\Delta} (p, q; t) \, dt.
	\end{split}
\end{equation}

The functions $\Gamma_{s}^{T}(x, y)$ and $\Gamma_{s}^{\Delta}(p, q)$ are called the heat kernels and they are the Schwartz kernels for the operators $e^{-sT}$ and $e^{-s\Delta}$ respectively (see \cite[p. 181]{BH78}). Note that $\Gamma_{s}^{\Delta}(p, q)$ coincides with the fundamental solution to the heat equation on $M$. The parametrix construction for the fundamental solutions of heat equations, which was first introduced in \cite{MR31145}, immediately implies the following estimate:
\begin{proposition} \label{P:HeatKernelEstM}
There exists positive constants $C_1$ and $C_2$ such that for any pair of points $p, q \in M^n$,
\[
	|\Gamma_s^{\Delta} (p, q) |\leq C_1s^{-\frac{n}{2}} \,e^{-C_2 \frac{d_M^2(p, q)}{s}}.
\]
\end{proposition}
\begin{proof}
The proposition directly follows from equation (45) on p. 154 of \cite{MR768584}.
\end{proof}

\subsection{Tauberian and Abelian Theorem}
Given two real-valued functions $f(t)$ and $g(t)$. We write $f(t)\sim g(t)$ as $t\to \infty$ (resp. $t\to 0$) if 
$\lim_{t\to \infty} f(t)/g(t)=1$ (resp. $\lim_{t\to 0} f(t)/g(t)=1$). We state the Tauberian theorem and Abelian theorem together below.
\begin{theorem}\cite[Theorem 2 on p. 445]{Feller}
\label{T:Tauberian} 
Let $F\colon  [0, \infty)\to \mathbb{R}$ be a function with bounded variation such that Laplace transform $f(s)=\int_0^\infty e^{-st}F(t)dt$ exists. 
Then for fixed $\alpha\in [0,\infty)$ and $C\in \mathbb{R}$,
 each of the relations
\begin{equation} \label{eqn:LaplaceF}
sf(s)\sim {C\over s^{\alpha}}\qquad s\to 0
\end{equation}
and
\begin{equation}\label{eqn:F}
F(t)\sim {C\over \Gamma(\alpha+1)}t^\alpha\qquad t\to \infty
\end{equation}
implies the other.
\end{theorem}
Historically, the implication from \eqref{eqn:F} to \eqref{eqn:LaplaceF} is called an Abelian theorem, and the converse is called a Tauberian theorem \cite[p. 445]{Feller}.

\section{Geometric Estimate} \label{S:geometricestimate}
The main goal of this section is to prove Theorem \ref{MT:geometricestimate} by induction on the dimension of $M$. For such an inductive argument to work, we will, in fact, prove a more general statement (Theorem \ref{T:geometricestimate} below), that applies to \emph{disconnected} singular Riemannian foliations. First, we need the following definition.

\todo{New definition added}
\begin{definition} \label{D:lift-iso}
Given a closed singular Riemannian foliation $(M, \F)$, an isometry of $M/\F$ is called \emph{liftable} if it is induced by an isometry of $M$ that permutes the leaves of $\F$.
\end{definition}
\todo{Definition changed}
\begin{definition} \label{D:disco}
Let $(M,\F_0)$ be a closed singular Riemannian foliation with canonical projection $\sigma_0: M\to M/\F_0$, and let $\Gamma$ be a finite group acting on $M/\F_0$ by liftable isometries. The \emph{disconnected singular Riemannian foliation} (with closed leaves) on $M$ induced by $(\F_0, \Gamma)$ is the partition $\F$ of $M$ into sets of the form $\sigma_0^{-1}(\Gamma \cdot x)$, for $x\in M/\F_0$, which are called the \emph{leaves of $\F$}.
\end{definition}
Given a singular Riemannian foliation with disconnected leaves as above, we observe that its leaf space is $M/\F=(M/\F_0)/\Gamma$, and the canonical projection $\sigma:M\to M/\F$ is the composition of $\sigma_0:M\to M/\F_0$ with the quotient map $M/\F_0\to (M/\F_0)/\Gamma$.
\todo{Example added}
\begin{example}
Given a compact disconnected Lie group $G$ acting on a Riemannian manifold $M$ by isometries, the partition into $G$-orbits is an example of disconnected singular Riemannian foliation. In this case $\F_0$ is the orbit decomposition into $G_0$-orbits, where $G_0$ is the identity component of $G$, $\Gamma=G/G_0$ and the action of $G/G_0$ on $M/G_0$ is the natural one.
\end{example}

\begin{remark}
\begin{itemize}
\item While the isometries in $\Gamma$ are required to be liftable, we do not assume that the entire action of $\Gamma$ lifts to an action on $M$. This is in general not the case even in the example above.
\item Definition \ref{D:disco} is very similar to \cite[Definition 8]{MR19}, with the exceptions that the latter allows for $\F$ and $\F_0$ to have non-closed leaves, and the total space is assumed to be Euclidean space.
\item 
The key fact we will use below is that any slice foliations of any closed singular Riemannian foliation is a disconnected singular Riemannian foliation in the sense of Definition \ref{D:disco}, see Section \ref{SS:slice}.
\end{itemize}
\end{remark}

The following theorem generalizes Theorem \ref{MT:geometricestimate}:

\begin{theorem}\label{T:geometricestimate}
Given a compact connected Riemannian manifold $M$ and a disconnected singular Riemannian foliation $\F$ with closed leaves, there exists a constant $C=C(M,\F)$ such that, for every $p\in M$, and every $s>0$, one has 
\begin{equation} \label{E:geometricestimate}
\int_{L_p}\exp\left(-{d_M(p,q)^2\over s}\right)dq\leq Cs^{\dim L_p\over 2},
\end{equation}
where the integral is computed with respect to the Riemannian volume form of $L_p$.
\end{theorem}

\subsection{Reductions}
To make the proof of Theorem \ref{T:geometricestimate} shorter, we prove a couple of reductions as separate lemmas, starting with the reduction to the case where the leaves are connected:
\begin{lemma}\label{L:disco}
Let $M,\F$ be as in Theorem \ref{T:geometricestimate}, and $\F_0, \Gamma$ as in Definition \ref{D:disco} .
Suppose  the conclusion of Theorem \ref{T:geometricestimate} holds  for $(M, \F_0)$. Then it holds for $(M,\F)$.\end{lemma}
\begin{proof}
We are assuming the existence of $C=C(M,\F_0)$ such that, for all $p\in M$, and all $s>0$,
\begin{equation} \label{E:F_0}
\int_{L_p}\exp\left(-{d_M(p,q)^2\over s}\right)dq\leq Cs^{\dim L_p\over 2},
\end{equation}
where $L_p$ denotes the $\F_0$-leaf through $p$. We claim that, for $(M,\F)$, we can take $C(M,\F)=C |\Gamma| 2^{\dim M}$.

Indeed, let $s>0$, $p\in M$, and let $L$ denote the $\F$-leaf through $p$. Then $L$ is the disjoint union of $\F_0$-leaves $L_1, \ldots L_k$, where $k\leq |\Gamma|$, and we may assume $p\in L_1$. Note that each $L_j$ is the image of $L_1$ under an isometry of $M$, so, in particular, they all have the same dimension, which we may denote by $\dim L$. Let $p_1=p$, and $p_j$ be a point of $L_j$ that is closest to $p$. Then, for any $q\in L_j$, we have, by the triangle inequality,
\[ d_M(q, p)\geq \frac{d_M(q, p)+d_M(p, p_j) }{2}\geq \frac{d_M(q, p_j)}{2}. \]
Therefore
\[\int_{L}\exp\left(-{d_M(p,q)^2\over s}\right)dq\leq \sum_{j=1}^k \int_{L_j}\exp\left(-{d_M(p_j,q)^2\over 4s}\right)dq
\leq |\Gamma| C (4s)^{\dim L/2}\]
where in the last inequality we have used \eqref{E:F_0} for each $L_j$, and $k\leq|\Gamma|$. Since $(4s)^{\dim L/2}\leq 2^{\dim M} s^{\dim L/2}$, this finishes the proof of the claim.
\end{proof}

The next lemma corresponds to the case in Theorem \ref{T:geometricestimate} where $\F$ has a single leaf:
\begin{lemma}\label{L:single}
Let $M$ be a compact connected Riemannian manifold. Then, there exists a constant $C$ such that, for every $p\in M$, and every $s>0$, one has 
\[
\int_{M}\exp\left(-{d_M(p,q)^2\over s}\right)dq\leq Cs^{\dim M\over 2}.
\]
\end{lemma}
\begin{proof}
See \cite[Lemma 4.5]{BH78}
\end{proof}


\begin{remark}[scaling-invariance] \label{R:scaling}
If Theorem \ref{T:geometricestimate} holds for a given manifold $(M,g)$, then it holds for $(M,\lambda^2 g)$, for any $\lambda>0$, with the \emph{same} constant $C$.
\end{remark}

\subsection{Sasaki metrics} \label{SS:Sasaki}
We will need to put a special metric on the normal bundle of a leaf $L$ in order to prove Lemma \ref{L:technical} below, so we collect here the definition and some basic properties.

\begin{definition}
Let $L$ be a closed Riemannian manifold, and $\pi\colon E\to L$ be a vector bundle equipped with an inner product $\langle,\rangle_x$ on the fiber $E_x=\pi^{-1}(x)$ for each $x\in L$, and a metric connection $\nabla$. Parallel translation determines a choice of ``horizontal space'' $H_{(x,v)}$ at each $(x,v)\in E$ (where $x\in L$ and $v\in E_x$), and we define the \emph{Sasaki metric} on $E$ by declaring $H_{(x,v)}$ to be orthogonal to the ``vertical space'' $V_{(x,v)}= \ker d\pi_{(x,v)}=T_{(x,v)}E_x=E_x$, using the given inner product on $V_{(x,v)}=E_x$, and pulling back the given inner product on $T_xL$ to $H_{(x,v)}$. \end{definition}

The construction above is also described in \cite[Example 9.60]{Besse}.

\begin{proposition} \label{P:Sasaki}
In the notation of the definition above, endow $E$ with the Sasaki metric. Then 
\begin{enumerate}[(a)]
\item $\pi\colon  E \to L$ is a Riemannian submersion;
\item For each $x\in L$, the fiber $E_x$ is flat and totally geodesic.
\item $E$ is complete.
\item There exists $r>0$ such that, for all $x\in L$,  one has $d_E=d_{E_x}$ on the set $E_x^{\leq r}= \{ v\in E_x\mid \|v\|\leq r\} $.
\end{enumerate}
\end{proposition}
\begin{proof}
\begin{enumerate}[(a)]
\item follows immediately from the definition.
\item For each $x\in L$, the Riemannian metric on $E_x$ is flat by definition, being isometric to the Euclidean space $(E_x, \langle,\rangle_x)$. For the second claim, see \cite[Theorem 9.59 and Example 9.60]{Besse}.
\item See \cite[Theorem 9.59 and Example 9.60]{Besse}.
\item Identify $L$ with the zero section in $E$, and note that, for every $r>0$, the closed tube of radius $r$ around $L$ in $E$ coincides with $\cup_{x\in L} E_x^{\leq r}$. Since $E$ is complete, the injectivity radius function on $E$ is continuous (see \cite[Theorem III.2.1]{Chavel}), and thus, by compactness of $L$, there exists $r>0$ such that the injectivity radius is larger than $2r$ on $\cup_{x\in L} E_x^{\leq r}$. 

Let $x\in L$, and $v,w\in E_x^{\leq r}$. Then $\|v-w\|\leq 2r$, and, in particular, $d_E((x,v), (x,w))\leq 2r$. Thus there is a unique $\xi\in T_{(x,w)}E$ with $\|\xi\| \leq 2r$ which exponentiates to $(x,v)$. Since $v-w$ is another vector with the same property (because $E_x$ is totally geodesic), we obtain $\xi= v-w$. Therefore $d_E((x,v), (x,w))=\|\xi\|=\|v-w\|=d_{E_x}(v,w)$. \qedhere
\end{enumerate}
\end{proof}

\subsection{Uniform estimates}
\todo{Title changed, phrasing changed}
The next two technical lemmas will be used in the proof of Theorem \ref{T:geometricestimate}, by providing uniform information about the geometry of leaves around a fixed one. We start with an analogue of \cite[Lemma 4.7]{BH78}:
\begin{lemma} \label{L:technical}
Let $L$ be a leaf of the closed singular Riemannian foliation $\F$ on the closed Riemannian manifold $M$. Assume $L$ has codimension at least one. Denote by $\pi \colon B_\epsilon(L)\to L$ the closest point projection, where $\epsilon$ is smaller than the normal injectivity radius of $L$. Let $E=\nu L$ denote the normal bundle of $L$ in $M$, and by $\exp$ the normal exponential map $\nu L\to M$. Then there exist $C_1>0$ and $r\in (0,\epsilon)$ such that, for every $y\in L$ and $p\in B_r(L)$, there exists $v=v(p,y)\in\nu_y L$, such that $\exp(v)\in L_p$, and, for all $q\in \pi^{-1}(y) \cap L_p$, one has
\[C_1 d^2_M(p,q)\geq d^2_L(\pi(p), \pi(q)) + d^2_{sph}(v,\exp^{-1}(q)) \]
if $\codim(L)\geq 2$, and
\[C_1 d^2_M(p,q)\geq d^2_L(\pi(p), \pi(q)) \]
if  $\codim(L)=1$. Here $d_{sph}$ denotes the distance in the round sphere in $\nu_y L$ of radius equal to the distance between $L$ and $L_p$ (that is, the leaf through $p$).
\end{lemma}
\begin{proof}
To simplify notation, denote by $\psi$ the inverse of the diffeomorphism
\[\exp|_{E^{< \epsilon}}\colon E^{< \epsilon}\to B_\epsilon(L).\]

Fix a metric connection on $E$ (e.g., the normal connection), and endow $E$ with the associated Sasaki metric, see Section \ref{SS:Sasaki}. Since $\epsilon$ is strictly smaller than the normal injectivity radius, there exists $K>0$ such that, for all $\eta\leq\epsilon$, the map $\exp|_{E^{< \eta}}\colon  E^{< \eta}\to B_\eta(L)$ is $K$-bi-Lipschitz.

Choose $r>0$ that is smaller than the ``$r$'' given in Proposition \ref{P:Sasaki}(d), and  such that $2r<\epsilon$.

We claim that, for all $p,q\in B_r(L)$,
\[ d_M(p,q) \geq C d_E (\psi(p), \psi(q)), \]
where $C= (2r+\diam (L))^{-1} K^{-1} r$. We consider two cases.

{\bf Case 1:} $d_M(p,q)\geq r$. Using the triangle inequality, we obtain
\[d_{B_r(L)}(p,q) \leq 2r+ \diam(L).\]
Thus 
\begin{align}
d_M(p,q) &\geq  r\cdot \frac{2r+\diam(L)}{2r+\diam(L)}\geq \frac{r}{2r+\diam(L)}d_{B_r(L)}(p,q)
                \nonumber \\
     &\geq \frac{r}{2r+\diam(L)} \frac{1}{K} d_{E^{\leq r}}(\psi(p), \psi(q)) \nonumber \\
     & \geq \frac{r}{2r+\diam(L)} \frac{1}{K} d_E(\psi(p), \psi(q)) = C d_E (\psi(p), \psi(q)). \nonumber
\end{align}

{\bf Case 2:} $d_M(p,q)< r$. Then $q\in B_r(p)\subset B_\epsilon(L)$, so that
\begin{align}
d_M(p,q) &= d_{B_\epsilon(L)}(p,q)\geq K^{-1} d_{E^{\leq \epsilon}}(\psi(p),\psi(q))
                \nonumber \\
     &\geq K^{-1} d_E(\psi(p),\psi(q))\geq C d_E(\psi(p),\psi(q)) \nonumber 
\end{align}
which finishes the proof of the Claim.

The Claim above reduces the proof of the statement of the Lemma to a similar statement, where ``$d_M(p,q)$'' is replaced with  ``$d_E(\psi(p),\psi(q))$''. 

If  $\codim(L)=1$, the desired inequality follows from the fact (see Proposition \ref{P:Sasaki}(a)) that $\pi\colon  E\to L$ is a Riemannian submersion. Assume  $\codim(L)\geq 2$.

Let $y\in L, \ p\in B_r(L)$. Choose $v$ to be a point of $E_y\cap \psi(L_p)$ minimizing $d_E(v,\psi(p))$. Let $q\in\pi^{-1}(y)\cap L_p$, which is equivalent to $\psi(q)\in E_y\cap \psi(L_p)$. See Figure \ref{F:diagram}.

\begin{figure}[!htb]
\begin{tikzpicture}
\draw (-1.5,0)--(-0.5,0);
\draw[dashed] (1.5,0)--(2,0);
\draw (2,0)--(4.2,0);
\draw (4.3,0)--(7,0);
\draw node at (7,-0.5) {$L$};

\draw (2.75,-0.05)--(2.75,-1.5)--(1.25,-2)--(1.25,1.5)--(2.75,2)--(2.75,0.05);
\draw node at (3,-1.2) {$E_y$};
\filldraw (2,0) circle (0.05) node[anchor=north]{$y$};

\draw (0,1)--(2,1);
\draw (0,-1)--(2,-1);
\draw (2,0) ellipse (0.5 and 1);
\draw (0,1) arc (90:270:0.5 and 1);
\draw node at (0.25,1.25) {\tiny $\psi(L_p)=\psi(L_q)$};


\draw (5.75,-0.05)--(5.75,-1.5)--(4.25,-2)--(4.25,1.5)--(5.75,2)--(5.75,0.05);
\draw node at (6.2,-1.2) {$E_{\pi(p)}$};

\filldraw (2,1) circle (0.05) node[anchor=south]{$v$};
\filldraw (5,1) circle (0.05) node[anchor=south]{$\psi(p)$};
\filldraw (5,0) circle (0.05) node[anchor=north]{$\pi(p)$};
\filldraw (1.6,0.6) circle (0.05) node[anchor=west]{\tiny $\psi(q)$};

\end{tikzpicture}
\caption{}\label{F:diagram}
\end{figure}

By the triangle inequality,
\[ d_E(\psi(q), v)\leq d_E(\psi(q), \psi(p))+ d_E(\psi(p),v).\]
By the choice of $v$, 
\[d_E(\psi(p),v)\leq d_E(\psi(q), \psi(p)).\]
Using Proposition \ref{P:Sasaki}(d), and some Euclidean Geometry, we obtain:
\[d_E(\psi(q), v)=d_{E_y}(\psi(q), v) \geq \frac{2}{\pi} d_{sph}(\psi(q), v).\]
Combining the last three inequalities, we obtain
\begin{equation} \label{E:sph}
d_E(\psi(q), \psi(p)) \geq \frac{1}{\pi} d_{sph}(\psi(q), v).
\end{equation}

On the other hand, since $\pi\colon  E\to L$ is a Riemannian submersion (see Proposition \ref{P:Sasaki}(a)), we have
\[   d_E(\psi(q), \psi(p))\geq d_L(y, \pi(p)).\]
Squaring this inequality, and adding the square of \eqref{E:sph}, we obtain
\[ 2 d^2_E(\psi(q), \psi(p))\geq  d^2_L(y, \pi(p)) + \frac{1}{\pi^2} d^2_{sph}(\psi(q), v),\]
from which the desired inequality follows (with $C_1=2\pi^2/C^2$).
\end{proof}

The closest-point projection to a leaf, when restricted to a nearby leaf,  is known to be a submersion (see Section \ref{SS:slice}). The next Lemma states that it is, in fact, not too far from being a \emph{Riemannian} submersion (cf. \cite{Lytchak23}, Proposition 1.3, for a more precise and general statement). It will be used to control the Jacobian appearing in a co-area formula inside the proof of Theorem \ref{T:geometricestimate} below.
\begin{lemma} \label{L:Jacobian}
Let $L$ be a leaf of the closed singular Riemannian foliation $\F$ on the closed Riemannian manifold $M$. Assume $L$ has codimension at least one, and denote by $\pi \colon  B_r(L)\to L$ the closest point projection, where $r>0$ is strictly smaller than the normal injectivity radius of $L$.
For each $q\in B_r(L)$, let $L_q$ denote the leaf through $q$, and  $d\pi' _q:T_qL_q\to T_{\pi(q)} L$ the restriction of $d\pi_q$ to $T_qL_q$. Let $f(q)$ be the smallest eigenvalue of the self-adjoint linear map 
\[d\pi' _q \circ (d\pi' _q)^*\colon  T_{\pi(q)} L\to T_{\pi(q)} L.\]
Then $C_2=\inf_{q\in B_r (L)} f(q)$ is positive.
\end{lemma}
\begin{proof}
Denote by $R$ the normal injectivity radius of $L$. Recall from Section \ref{SS:slice} that, for every $q\in B_R(L)$, the map $\pi|_{L_q}\colon L_q\to L $ is a submersion, in particular $f(q)>0$.

Assume, for a contradiction, that $C_2=0$. Then, there exists a sequence $q_i\in B_r(L)$ with $\lim f(q_i)=0$.  Choose $v_i\in T_{\pi(q_i)}L$ to be a unit eigenvector of $d\pi' _{q_i} \circ (d\pi' _{q_i})^*$ associated to the eigenvalue $f(q_i)$. 

By compactness, we may assume, after passing to a subsequence, that $q_i$ converges to $q\in \overline{B_r(L)}\subset B_R(L)$, and that $v_i$ converge to $v\in T_{\pi(q)}L$.

Since $d\pi' _{q}$ is surjective, let $w\in T_qL_q$ such that $d\pi' _{q}(w)=v$. Since $\F$ is a singular Riemannian foliation, it is generated by smooth vector fields. In particular, there exist vectors $w_i\in T_{q_i}L_{q_i}$ converging to $w$, so that $\lim d\pi_{q_i}(w_i)= d\pi_q(w)=v$. Then
\[\lim \langle (d\pi'_{q_i})^*v_i, w_i\rangle =\lim \langle v_i, d\pi_{q_i} w_i\rangle=\langle v, v\rangle=1.  \]

On the other hand, 
\[ \lim \|  (d\pi'_{q_i})^*v_i \|^2= \lim \langle d\pi'_{q_i}\circ (d\pi'_{q_i})^*v_i, v_i\rangle =\lim f(q_i)=0\]
and $\lim \|w_i\| =\|w\| <\infty $, so that, by the Cauchy-Schwarz inequality, 
\[ \lim \langle (d\pi'_{q_i})^*v_i, w_i\rangle =0,\]
a contradiction. Therefore $C_2>0$.
\end{proof}

\subsection{Proof of Theorem \ref{T:geometricestimate}}


\begin{proof}[Proof of Theorem \ref{T:geometricestimate}]
Let $\F_0$ and $\Gamma$ as in Definition \ref{D:disco}.

We use induction on $\dim M$. If $\dim M=1$, then either $\F=\F_0$ has a single leaf, in which case we can apply Lemma \ref{L:single}, or $\F_0$ is the trivial foliation by points, for which the conclusion is trivial, so we can apply Lemma \ref{L:disco}. Assume $\dim M>1$.

By Lemma \ref{L:disco}, we may assume that $\F=\F_0$, that is, $\F$ is a singular Riemannian foliation (with connected, closed leaves).

By Lemma \ref{L:single}, we may assume that $M$ is not a single leaf. Since $M$ is connected, that means every leaf has codimension at least one.

Since $M$ is compact, it is enough to prove that, for every leaf $L$, there exists $r, C_L>0$ such that \eqref{E:geometricestimate} holds for all $p$ in the tube $B_r(L)$ and all $s>0$, with $C_L$ in place of $C$. Indeed, $M$ can be covered by finitely many such tubes, so one can take $C$ to be the maximum of the corresponding constants $C_L$. Therefore, we fix a leaf $L$.

Take $r>0$ and $C_1$ given by Lemma \ref{L:technical} (in particular $r$ is smaller than the normal injectivity radius of $L$), and $C_2$ given in Lemma \ref{L:Jacobian}.  Let $p\in B_r(L)$. The restriction of $\pi$ to $L_p$ is a submersion $L_p\to L$ (see Section \ref{SS:slice}), with Jacobian 
\[\operatorname{Jac}(q)=\sqrt{\det \left(d\pi'_q\circ (d\pi'_q)^* \right)}\]
which, by Lemma \ref{L:Jacobian}, satisfies $ \operatorname{Jac}(q)\geq C_2^{\dim L /2 }$ for all $q\in L_p$. To simplify notation, let
\[g_s(q)=\exp\left(-{d_M(p,q)^2\over s}\right).\]
Then
\[\int_{q\in L_p}g_s(q)dvol^{L_p}(q)\leq C_2^{-\dim L /2 }\int_{q\in L_p}g_s(q)\operatorname{Jac}(q) dvol^{L_p}(q).\]
Applying the coarea formula \cite[p. 160, Exercise III.12(c)]{Chavel}, the integral on the right becomes
\begin{equation} \label{E:coarea}
\int_{y\in L} \left[ \int_{q\in L_p\cap \pi^{-1}(y)} g_s(q) dvol^{L_p\cap \pi^{-1}(y)}(q) \right] dvol^{L}(y)
\end{equation}

The next step is to use the inequality in Lemma \ref{L:technical}, so we split in two cases: $\codim(L)=1$ and $\codim(L)\geq 2$.

{\bf Case 1: $\codim(L)=1$.} By Lemma \ref{L:technical}, $g_s(q)\leq \exp(-d^2_L(\pi(p), \pi(q))/ C_1s)$. Note that the leaves $L$ and $L_p$ have the same dimension
and, for each $y\in L$, the set $L_p\cap \pi^{-1}(y)\subseteq \{\exp_y x\mid x\in \nu_yL, \|x\|=d_M(y,p)\}$ consists of at most two points. Thus \eqref{E:coarea} is less than or equal to
\[2\int_{y\in L}  \exp\left( \frac{-d^2_L(\pi(p), y)}{C_1s} \right) dvol^{L}(y)\]
which, by Lemma \ref{L:single}, is at most 
\[2C_3 (C_1s)^{\dim L\over 2} = 2C_3 C_1^{\dim L_p\over 2} s^{\dim L_p\over 2} \]
for some $C_3$ that depends only on $L$. Therefore, 
\[ \int_{L_p}\exp\left(-{d_M(p,q)^2\over s}\right)dq\leq \left[C_2^{-\dim L /2 } 2C_3 C_1^{\dim L_p\over 2}\right] s^{\dim L_p\over 2}, \]
finishing the proof of \eqref{E:geometricestimate} in Case 1.
 
{\bf Case 2: $\codim(L)\geq2$.} By Lemma \ref{L:technical},
\[g_s(q) \leq \exp\left( \frac{-d^2_L(\pi(p),\pi(q))}{C_1s} \right) \exp\left( \frac{-d^2_{sph}(v,\exp^{-1}(q))}{C_1s} \right)\]
where $v=v(p,y)$ depends only on $p$ and $y$.

Thus \eqref{E:coarea} is at most
\begin{equation}\label{E:coarea2}
\int_{y\in L}  e^{ \frac{-d^2_L(\pi(p),y)}{C_1s} } \left[ \int_{q\in L_p\cap \pi^{-1}(y)} e^{ \frac{-d^2_{sph}(v,\exp^{-1}(q))}{C_1s} } dvol^{L_p\cap \pi^{-1}(y)}(q) \right] dvol^{L}(y)
\end{equation}
Next we make a change of variables in the inner integral above, using the normal exponential map $\exp\colon \nu_y L\to \pi^{-1}(y)$. This is a diffeomorphism when restricted to the ball of radius $r$ in $\nu_y L$, and, since $L$ is compact, it is bi-Lipshiptz with a constant depending only on $L,r$. Since $dvol^{L_p\cap \pi^{-1}(y)}$ is the Riemannian volume form of the Riemannian manifold $L_p\cap \pi^{-1}(y)$ endowed with the induced Riemmanian structure from $\pi^{-1}(y)$, there exists $C_4$, depending only on $L,r$, such that the inner integral in \eqref{E:coarea2} is at most
\begin{equation}\label{E:infinitesimal}
C_4\int_{z\in L_v} e^{ \frac{-d^2_{sph}(v,z)}{C_1s} } dvol^{L_v}(z).
\end{equation}
Here $L_v= \exp^{-1}(L_p\cap \pi^{-1}(y))$ is the leaf through $v$ of the (disconnected) slice foliation $\F^y$ of the Euclidean space  $\nu_yL$, see Section \ref{SS:slice}.

We apply the inductive hypothesis to the restriction of $\F^y$ to the unit sphere $\sphere(\nu_y L)$ and conclude that there exists $C_5$, such that, for all $w\in \sphere(\nu_y L)$,
\begin{equation} \label{E:induction}
\int_{z\in L_w} e^{ \frac{-d^2_{sph}(w,z)}{C_1s} } dvol^{L_w}(z)\leq C_5 (C_1 s)^{\dim L_w/2}
\end{equation}
By Remark \ref{R:scaling}, the same conclusion holds with $w$ replaced with any $v\in\nu_y L$ (and with the same constant $C_5$), because homothetic transformations of $\nu_y L$ take leaves of $\F^y$ to leaves of $\F^y$ (see second paragraph after \cite[Definition 8]{MR19}). 
In principle, the constant $C_5$ depends on $y\in L$, but in fact that it can be chosen independently of $y$, because the slice foliation does not depend on $y$, see Section \ref{SS:slice}.

Putting together the inequality \eqref{E:coarea} $\leq$ \eqref{E:coarea2} with the estimates \eqref{E:infinitesimal} and \eqref{E:induction}, and noting that $\dim L_v +\dim L= \dim L_p$, we obtain
\[
\int_{q\in L_p}g_s(q)dvol^{L_p}(q)\leq 
 C_2^{-\dim L /2 } C_4 C_5 (C_1 s)^{(\dim L_p-\dim L)/2} \int_{y\in L}  e^{ \frac{-d^2_L(\pi(p),y)}{C_1s} }   dvol^{L}(y)
\]
Finally, by Lemma \ref{L:single}, there exists $C_3$ (depending only on $L$), such that the integral on the right is at most $C_3 (C_1s)^{\dim L /2}$. Thus
\[
\int_{q\in L_p}g_s(q)dvol^{L_p}(q)\leq 
 \left(C_2^{-\dim L /2 } C_4 C_5 C_3 C_1^{\dim L_p /2} \right)s^{\dim L_p/2}\]
which concludes the proof of \eqref{E:geometricestimate} in Case 2.
\end{proof}

\section{Weyl's Law} \label{S:Weyl}

Throughout this section, we will use the terminology introduced in Section \ref{S:prelim}.

We follow the approach in \cite[Section 3]{BH78} to prove Weyl's Law (Theorem \ref{MT:Weyl}).
\todo{Changed sketch of section, adding details}
The proof is roughly organized as follows. First, we derive a trace formula for 
the operator $e^{-sT}$, where $T: L^2(X_0,h)\to L^2(X_0,h)$ is the operator defined in Section \ref{SS:Basic-Lap}. This formula relates the Laplace transform of the counting function $N(t)$ of $T$, to an integral of its heat kernel. At each point in $X_0$, one can apply H\"ormander's estimate (Theorem \ref{T:Hormander}) and the Abelian theorem (Theorem \ref{T:Tauberian}) to obtain the asymptotics for the heat kernel. Theorem \ref{MT:geometricestimate} allows us to integrate the asymptotics of the heat kernel. The asymptotics for $N(t)$ then follows from the Tauberian theorem and the trace formula.

\subsection{A trace formula}

The counting function $N(t)$ and $\Gamma_{s}^{T}$ are related via the following trace formula for $e^{-sT}$:

\begin{proposition} \label{P:TraceFormula}
The function $N(t)$ has polynomial growth as $t\to \infty$. Furthermore,
\begin{equation} \label{eqn:Trace}
	s\int_{0}^{\infty} e^{-st}N(t)\, dt=\int_{X_0} \Gamma_s^{T} (x, x) h(x) \, dx,
\end{equation}
where $\Gamma_s^{T}(x, y)$ is the heat kernel of $T$ defined by \eqref{eqn:HeatKernelDef}. 
\end{proposition}
\begin{proof}
Let $S(t)$ be the number of eigenvalues of the Laplace operator $\Delta$ on $M$ that are less than $t$.
The standard Weyl's Law for the Laplace operator on compact Riemannian manifolds implies that $S(t)$ must have polynomial growth as $t\to \infty$. Since the growth rate of $N(t)$ cannot exceed that of $S(t)$, $N(t)$ must have polynomial growth as $t\to \infty$.

To prove the integral formula, let $\{\varphi_i\}$ be the orthonormal basis for $L^2(X_0, h)$ in Remark \ref{R:Eigenbasis} such that  $T(\varphi_i)=\lam_i\varphi_i$. Then by Proposition \ref{P:SpectralFunctionExp}(a),
\begin{equation}
N(t)=\sum_{\lam_i <t} ||\varphi_i(x)||_{L^2(X_0, h)}^2=\int_{X_0} \theta^{T}(x, x, t) \, h(x) \,dx.
\end{equation}
Hence 
\begin{equation}
s\int_{0}^{\infty} e^{-st}N(t)\, dt=s \int_{0}^{\infty}  e^{-st} \int_{X_0} \theta^{T}(x, x, t) \, h(x) \,dx \, dt.\\
\end{equation}

Since $\theta^{T}(x, x, t)$ has uniform polynomial growth, (cf.  Section \ref{SS:heat-K}) we may use Fubini's Theorem to interchange the integral and obtain
\begin{equation}
\begin{split}
s\int_{0}^{\infty} e^{-st}N(t)\, dt&= \int_{X_0} \int_{0}^{\infty} s e^{-st} \theta^{T}(x, x, t) \,dt\, h(x)\,dx\\
&=\int_{X_0} \Gamma_s^{T} (x, x) h(x) \, dx,
\end{split}
\end{equation}
which proves the integral formula.
\end{proof}

Proposition \ref{P:HormanderT} and Theorem \ref{T:Tauberian} directly imply the following asymptotic estimate for the heat kernel:
\begin{proposition} \label{P:KernelEst}
For each point $x \in X_0$, we have 
\[
	\Gamma_{s}^{T}(x, x)\sim s^{-m/2} \Gamma\left(\frac{m}{2}+1\right) \frac{\omega_{m}}{h(x)(2\pi)^m}
\]
as $s\to 0$.
\end{proposition}

\subsection{Uniform estimates for the heat kernel}

Proposition \ref{P:KernelEst} in the previous section can be restated by saying that the function $h(x)s^{m/2}\Gamma_s(x,x)$, as a function of $s$ and $x$, converges pointwise to a constant as $s\to 0$. Such convergence however might not be uniform, as $X_0$ is not compact in general. Nevertheless, the main result in this section uses Theorem \ref{MT:geometricestimate} to show that this function is uniformly bounded. This will allow us to use the Dominated Convergence Theorem which, together with the Tauberian Theorem, will give Weyl's law.

\begin{theorem}\label{T:UniformEstimate}
Let $M^n$ be a compact $n$-dimensional manifold and let $\F$ be a closed singular Riemannian foliation with basic mean curvature. Then there exists a constant $C$ such that for all $x\in X_0$,
\[|s^{m/2}\Gamma_{s}^{T}(x, x)h(x)|<C,\]
where $m=\dim(X_0)$.
\end{theorem}
To prove Theorem \ref{T:UniformEstimate}, we first relate the heat kernel of $\Delta$ to that of $T$.
\begin{lemma} \label{L:HeatKernelRelation}
Let $\Gamma_s^{T}$ and $\Gamma_s^{\Delta}$ be the heat kernel of $T$ and $\Delta$ constructed by \eqref{eqn:HeatKernelDef}. Then for each pair of points $p, q\in M_0$, we have
\[
	\Gamma_s^{T}(\sigma(p), \sigma(q))=(\Av \otimes \Av) \Gamma_s^{\Delta} (p, q).
\]
\end{lemma}
\begin{proof}
Let $\{\phi_k\}$ and $\{\varphi_i\}$ be the orthonormal basis of $L^2(M_0)$ and $L^2(X_0, h)$ constructed in Remark \ref{R:Eigenbasis} and let $\theta^T(x, y ; t)$ and $\theta^{\Delta}(x, y ; t)$
be the spectral functions of $T$ and $\Delta$ respectively.

Since $\theta^{T}(x, y; t)$ has uniform polynomial growth, Proposition \ref{P:SpectralFunctionExp} and Fubini's theorem imply
\begin{equation*}
\begin{split}
\Gamma_{s}^T(\sigma(p), \sigma(q)) &= s\int_{0}^{\infty} e^{-st} \theta^{T}(\sigma(p), \sigma(q) ; t) \, dt\\
&=s\int_{0}^{\infty} e^{-st} (\Av\otimes \Av) \theta^{\Delta}(p, q; t) \, dt\\
&=(\Av\otimes \Av) \Gamma_s^{\Delta} (p, q).
\end{split}
\end{equation*}
The lemma is proved.
\end{proof}

\begin{proof} [Proof of Theorem \ref{T:UniformEstimate}]
Let $\F$ be a closed singular Riemannian foliation with basic mean curvature and $\dim(X_0)=m$.
Let $x\in X_0$ and $p\in \sigma^{-1}(x)$. By Lemma \ref{L:HeatKernelRelation} and Proposition \ref{P:HeatKernelEstM},

\begin{equation*}
\begin{split}
|s^{m/2}\Gamma_{s}^{T}(x, x)h(x)| &= \left| \frac{s^{\frac{m}{2}}}{\vol(L_p)} \int_{L_p} \int_{L_{p}} \Gamma_s^{\Delta}(p',q') \,dp'\,dq'\right|\\
&\leq \sup_{q\in L_p} s^{\frac{m}{2}}\int_{L_p} |\Gamma_s^{\Delta} (p', q) |\, dp'\\
&\leq  C_1 \sup_{q\in L_p} s^{\frac{m-n}{2}} \int_{L_p} e^{-C_2 \frac{d^2_M(p', q)}{s}} \, dp'.
\end{split}
\end{equation*}
Since $\dim(L_p)=n-m$, the theorem follows from Theorem \ref{MT:geometricestimate}.
\end{proof}

\subsection{Proof of Weyl's law}
\begin{proof} [Proof of Theorem \ref{MT:Weyl}]
Since $X_0$ has finite volume, Proposition \ref{P:KernelEst}, Theorem \ref{T:UniformEstimate}, and the Dominated Convergence Theorem imply that 
\[
	\int_{X_0} \Gamma_s^{T} (x, x) h(x) \, dx \sim s^{-m/2} \Gamma\left(\frac{m}{2}+1\right) \frac{\omega_{m}}{(2\pi)^m} \vol (X_0).
\]
Theorem \ref{MT:Weyl} then follows from Proposition \ref{P:TraceFormula} (the trace formula) and Theorem \ref{T:Tauberian} (the Tauberian theorem).
\end{proof}

%
%

%

\section{Foliations of the sphere and Invariant Theory}\label{S:spherical}

In this section we consider the special case of closed singular Riemannian foliations $(M,\F)$, where $M$ is a round sphere. The main goal is to relate the volume of the leaf space to algebraic invariants of the algebra of basic polynomials (see Section \ref{SS:InvT}), culminating in the proof of Theorem \ref{MT:spherical}. This is an application of Weyl's Law (Theorem \ref{MT:Weyl}) and the theory of Spherical Harmonics. 

\subsection{The Cohen--Macaulay property and Hilbert series}

Let $V$ be a $n+1$-dimensional real vector space with inner product, and let $\R[V]$ denote the $\R$-algebra of real-valued polynomial functions on $V$. Fix an orthonormal basis for $V$, via which the algebra $\R[V]$ is isomorphic to $\R[x_0,\ldots x_n]$.

Following \cite[Definition 4]{MR20}:
\begin{definition}
A sub-algebra $A\subset \R[V]$ is called \emph{Laplacian} if it contains $r^2=\sum_i x_i^2$ and is preserved by the Laplace operator $\Delta=\sum_i \partial^2 / \partial x_i^2$.
\end{definition}
The typical examples of Laplacian algebras are rings of invariants $A=\R[V]^G$, where $G$ is a group acting on $V$ by linear isometries; and the ring of basic polynomials $A=\B(\F)$, where $\F$ is an infinitesimal singular Riemannian foliation of $V$. More generally, by \cite{MR20, MR23}, Laplacian algebras are in one-to-one correspondence with infinitesimal manifold submetries of $V$.

Following \cite[Definition 2.5.6]{DerksenKemper}:
\begin{definition}
Let  $A\subset \R[V]$ be a graded sub-algebra. A subset $\{f_0,\ldots f_m\}\subset A$ of homogeneous elements is called a \emph{homogeneous system of parameters} if they are algebraically independent, and $A$ is a finitely generated module over the sub-algebra $\R[f_0, \ldots, f_m]$. 
\end{definition}
We note that finitely generated graded sub-algebras $A\subset \R[V]$ always have a homogeneous system of parameters, by the Noether Normalization Lemma, see \cite[Corollary 2.5.8]{DerksenKemper}.

Following \cite[Proposition 2.6.3]{DerksenKemper}:
\begin{definition}
Let $A\subset \R[V]$ be a finitely generated graded sub-algebra. We say $A$ is \emph{Cohen--Macaulay} if, for  every homogeneous system of parameters $f_0, \ldots f_m$, the algebra $A$ is a \emph{free} module over $\R[f_0, \ldots, f_m]$.
\end{definition}
We note that, in the above definition, the quantifier ``for every'' can be replaced with ``there exists'', resulting in an equivalent definition (see  \cite[Proposition 2.6.3]{DerksenKemper}).

Our first result is the observation that the celebrated Hochster--Roberts Theorem (\cite[Theorem 2.6.5]{DerksenKemper}) generalizes from rings of invariants to arbitrary Laplacian algebras:
\begin{proposition} \label{P:CM}
Let $A\subset \R[V]$ be a Laplacian algebra. Then $A$ is graded, finitely generated, and Cohen--Macaulay.
\end{proposition}
\begin{proof}
By \cite[Corollary 22]{MR20}, the Laplacian algebra $A$ is graded, and by \cite[Lemma 24]{MR20}, it is finitely generated. By \cite[Theorem 2.6.11]{DerksenKemper}, it suffices to show that, for every ideal $I\subset A$, one has $(I\R[V]) \cap A = I$.

By \cite[Theorem 23]{MR20}, there exists a Reynolds operator $\Pi\colon \R[V]\to A$. This means $\Pi$ is a linear projection onto $A$, and $\Pi (fg)=f\Pi(g)$ for all $f\in A$ and $g\in\R[V]$. Let $I\subset A$ be an ideal. The inclusion $I\subset (I\R[V]) \cap A $ is clear. To prove the reverse inclusion, let $f\in (I\R[V]) \cap A$. Then $f=\sum_i g_i f_i$ for $f_i\in I$ and $g_i\in \R[V]$. Applying $\Pi$ to this equation, and using the assumption that $f\in A$, so that $\Pi(f)=f$, we obtain $f=\sum_i f_i \Pi(g_i)$, which proves that $f\in I$, because $I$ is an ideal of $A$, and  $\Pi(g_i)\in A$. Therefore $(I\R[V]) \cap A = I$, and $A$ is Cohen--Macaulay.
\end{proof}
We note that the proof that $(I\R[V]) \cap A = I$ above is analogous to the proof of \cite[Lemma 2.6.10]{DerksenKemper}.

\begin{definition} 
Given any graded algebra $A=\bigoplus_{k=0}^\infty A_k$, we define its Hilbert series by 
\[ H(z)= \sum_{k=0} ^\infty \dim(A_k) z^k.\]
\end{definition}

The next lemma is an application of Weyl's Law (Theorem \ref{MT:Weyl}) together with the theory of Spherical Harmonics to relate the volume of the leaf space to the growth of coefficients of a Hilbert series. 
\begin{lemma} \label{L:sphericalWeyl}
Let $\F$ be a closed infinitesimal singular Riemannian foliation of $V$, and let $X=\sphere V/\F$ the leaf space. Let $A\subset\R[V]$ be the associated Laplacian algebra of basic polynomials, and let $H(z)$ denote its Hilbert series. Then the coefficients of the power series 
\[(1+z)H(z)=\sum_{k=0}^\infty b_k z^k\]
satisfy 
\[ b_k \sim {vol(X)\omega_{m}\over (2\pi)^{m}}k^{m}\qquad \text{as } k\to \infty \]
where $m=\dim X$ and $\omega_{m}=vol \DD^{m}$.
\end{lemma}

\begin{proof}
The singular Riemannian foliation $\F$ has basic mean curvature, see Section \ref{SS:princ}. By Theorem \ref{MT:Weyl}, the counting function $N(t)$ for the basic spectrum satisfies
\begin{equation} \label{E:Weyl}
N(t)\sim {vol(X)\omega_{m}\over (2\pi)^{m}}t^{m/2}\qquad t\to \infty
\end{equation}

By the theory of spherical harmonics (see e.g. \cite[Ex 12(e), page 118]{HoweTan}), the eigenvalues of the sphere $\sphere V$ are $\{k(k+n-1) \mid k=0,1, \ldots\}$, with corresponding eigenspaces the restrictions to $\sphere V$ of the space $\mathcal{H}_k\subset \R[V]_k$ of all homogeneous harmonic polynomials of degree $k$. Moreover, one has the direct sum decomposition 
\[\R[V]_k= \mathcal{H}_k \oplus r^2\R[V]_{k-2}\]
 for all $k\geq 2$.
 
Since $A$ is Laplacian, we obtain the analogous decomposition
\[A_k= (A\cap \mathcal{H}_k) \oplus r^2 A_{k-2}\]
see \cite[Lemma 4(b)]{MR23}. Note that the basic eigenvalues are $k(k+n-1)$, with multiplicity $\dim(A\cap \mathcal{H}_k)$. Denoting by $G(z)$ the corresponding generating function
\[G(z)=\sum_{k=0}^\infty \dim(A\cap \mathcal{H}_k) z^k\]
we obtain the formula $G(z)=(1-z^2)H(z)$, and thus \[(1+z+z^2+ \cdots)\ G(z)=(1+z)H(z) =\sum_{k=0}^\infty b_k z^k.\]
Note that
\[b_k= \sum_{j=0}^k \dim(A\cap \mathcal{H}_k) = N(k(k+n-1)).\]
Using \eqref{E:Weyl}, we obtain the desired asymptotic.
\end{proof}

The next lemma is purely algebraic, and will be used in combination with Lemma \ref{L:sphericalWeyl} in the proof of Theorem \ref{MT:spherical}. Moreover, part (b) gives algebraic constraints on the Hilbert series of the algebra of basic polynomials of an infinitesimal singular Riemannian foliation. These constraints seem to be new,
and may be of independent interest.

\begin{lemma} \label{L:asymptotics}
Let $A\subset\R[V]$ be a graded, finitely generated, Cohen--Macaulay algebra, and denote its Hilbert series by $H(z)$. Let $f_0\ldots f_m\in A$ be a homogenous system of parameters (so $m+1$ is the Krull dimension of $A$). Assume that the coefficients of the power series 
\[(1+z)H(z)=\sum_{k=0}^\infty b_k z^k\]
satisfy 
\[ b_k \sim C k^{m'}\qquad \text{as } k\to \infty \]
for some $C\neq 0$ and $m'\in\N$. Then 
\begin{enumerate}[(a)]
\item $m=m'$.
\item $H(z)$ is a rational function whose poles are roots of unity, such that $z=1$ is a pole of order $m+1$, $z=-1$ is (possibly) a pole of order at most $m+1$. All other poles have order at most $m$.
\item The Laurent series of $H(z)$ at $z=1$ has the form 
\[\frac{Cm!}{2}(1-z)^{-m-1} + (\cdots)(1-z)^{-m} + \cdots  \]
\item \[C= \frac{2\dim_{\R[f_0,\ldots, f_m]} A}  {m! \deg f_0 \cdots \deg f_m}\]
where $\dim_{\R[f_0,\ldots, f_m]} A$ denotes the rank of $A$ as a free, finitely generated module over $\R[f_0,\ldots, f_m]$.
\end{enumerate}
\end{lemma}
\begin{proof}
Since the algebra $A$ is Cohen--Macaulay, it is a finite free module over $\R[f_0,\ldots, f_m]$. Let $g_1, \ldots, g_\ell$ be a basis consisting of homogeneous elements of $A$, so that $\ell=\dim_{\R[f_0,\ldots, f_m]} A$. Thus
\[A=\R[f_0,\ldots, f_m] g_1 \oplus \cdots \oplus \R[f_0,\ldots, f_m] g_\ell\]
which implies that the Hilbert series $H(z)$ has the following form (sometimes called a ``Hironaka decomposition''):
\begin{equation}\label{E:Hironaka}
H(z)= \frac{z^{\deg(g_1)} + \cdots +z^{\deg(g_\ell)} }{(1-z^{\deg(f_0)}) \cdots (1-z^{\deg(f_m)})}
\end{equation}
This proves the first sentence of (b).

Let $d=\gcd(\deg(f_0), \ldots, \deg(f_m)) $, let  $\Omega\subset \C$ denote the set of all $d$th roots of unity, and let $\Omega'\subset \C$ denote the (finite) set of all poles of $(1+z) H(z)$.
For each $\omega\in \Omega'$, denote by $d(\omega)$ the order of $\omega$. Note that every $\omega\in\Omega'\setminus \Omega$ is a root of unity with order $d(\omega)\leq m$.


The partial fraction decomposition of $(1+z) H(z)$ is
\begin{equation}\label{E:Frdec}
(1+z) H(z)= P(z) + \sum_{\omega\in\Omega} \alpha_{\omega}(\omega-z)^{-m-1} + \sum_{\omega\in\Omega'}\sum_{j=1}^{d_\omega} \beta_{j,\omega}(\omega-z)^{-j}
\end{equation}
where $P(z)$ is the polynomial part of $(1+z)H(z)$ and $d_\omega=\min(d(\omega),m)$. Using the binomial series, we obtain
\begin{equation}\label{E:binseries}
(\omega-z)^{-j}= \omega^{-j} (1-(\omega^{-1} z))^{-j}= \omega^{-j}\sum_{k=0}^\infty \binom{k+j-1}{k} \omega^{-k} z^k
\end{equation}
\todo{added more detail expanding equations above}
Putting equations \eqref{E:Frdec} and \eqref{E:binseries} together, the formula for $(1+z) H(z)$ becomes
\begin{align*}
P(z) + \sum_k\left(\sum_{\omega\in\Omega} \alpha_{\omega}\omega^{-k-m-1}{k+m \choose k} + \sum_{\omega\in\Omega'}\sum_{j=1}^{d_\omega} \beta_{j,\omega}\omega^{-k-j}{k+j-1 \choose k}\right)z^k
\end{align*}

Noting that, for fixed $j$, the binomial $\binom{k+j-1}{k}$ is a polynomial in $k$ with leading term $k^{j-1}/(j-1)!$, we get
\begin{align*}
b_k\sim&\left(\sum_{\omega\in\Omega}{1\over m!} \alpha_{\omega}\omega^{-k-m-1}\right)k^{m} + \sum_{j=1}^{d_\omega}\left(\sum_{\omega\in\Omega'} {1\over (j-1)!}\beta_{j,\omega}\omega^{-k-j}\right)k^{j-1}\qquad \text{as }k\to \infty
\end{align*}
\todo{Added note here}
We note that the sums multiplying the powers of $k$ are periodic in $k$. If the first sum vanished identically, that is:
\[
\sum_{\omega\in\Omega}{1\over m!} \alpha_{\omega}\omega^{-k-m-1}=0\qquad \forall k
\]
then, averaging over all $k$ from $1$ to $d$, we would get:
\[
0={1\over d}\sum_{k=1}^{d}\sum_{\omega\in\Omega}{1\over m!} \alpha_{\omega}\omega^{-k-m-1}={1\over d}\sum_{\omega\in\Omega}{1\over m!} \alpha_{\omega}\omega^{-m-1}\left(\sum_{k=1}^{d}\omega^{-k}\right)={1\over m!}\alpha_1
\]
where the last equality follows from $\sum_{k=1}^{d}\omega^{-k} =0$ for $\omega\in\Omega\setminus \{1\}$. Since $\alpha_1\neq 0$, we thus know that the asymptotic behavior of $b_k$ as $k\to\infty$ is dominated by the term 
\[ \frac{1}{m!}\left(\sum_{\omega\in\Omega} \alpha_{\omega} \omega^{-m-1}\omega^{-k} \right)k^{m}\qquad \text{as } k\to\infty.\]

Our assumption that
\[ b_k \sim C k^{m'}\qquad \text{as } k\to \infty \]
implies that $m=m'$ (proving (a)), and that the periodic function
\[
\frac{1}{m!}\left(\sum_{\omega\in\Omega} \alpha_{\omega} \omega^{-m-1}\omega^{-k} \right)
\]
is constant (hence equal to its average) and equal to $C$:
\begin{equation} \label{E:periodic}
{1\over m!}\alpha_1={1\over m!}\sum_{\omega\in\Omega} \alpha_{\omega} \omega^{-m-1}\omega^{-k}  =C \qquad \forall k.
\end{equation}

Using the partial fraction decomposition \eqref{E:Frdec}, this implies (c). Similarly, fixing $\omega_0\in \Omega\setminus\{1\}$, multiplying \eqref{E:periodic} with $\omega_0^k$, and adding for $k=1, \ldots d$ yields $\alpha_{\omega_0}=0$, thus proving (b).

Finally, (d) follows from (c) and the Hironaka decomposition \eqref{E:Hironaka} of $H(z)$.
\end{proof}

\begin{remark} The arguments in the proof above are standard in enumerative combinatorics, see for example \cite[Chapter 5]{Wilf} for a general reference, and \cite{MOanswer} for the specific inspiration.
\end{remark}

\begin{remark}
We point out that the conclusions of Lemma \ref{L:asymptotics} do not hold for general Cohen-Macaulay algebras, for example $\mathbb{R}[t^3]$. The authors do not know, however, whether every Cohen-Macaulay algebra containing $r^2$ satisfies the conclusions of the Lemma.
\end{remark}

\subsection{Proof of Theorem \ref{MT:spherical} and its representation-theoretic analogue}
In this section we prove the following, slightly more precise version of Theorem \ref{MT:spherical}:

\begin{theorem}
\label{T:spherical}
Let $(V,\F)$, $\dim V=n+1$, be an infinitesimal foliation with closed leaves, let $A=\bigoplus_i A_i\subset\R[V]$ be the associated Laplacian algebra of basic polynomials, and let $H(z)=\sum_i (\dim A_i) z^i$ denote its Hilbert series. Denote by $\sphere V$ the unit sphere in $V$, and by $m$ the dimension of the leaf space $X=\sphere V /\F$.

Then:
\begin{enumerate}[(a)]
\item The Laurent series of $H(z)$ at $z=1$ has the form
\[H(z)= \frac{\Vol(X)}{\Vol(\sphere^m)} (1-z)^{-m-1} + (\cdots) (1-z)^{-m} + \cdots\]
\item $A$ is a Cohen--Macaulay algebra with Krull dimension $m+1$, and, for every homogeneous system of parameters $f_0, \ldots f_m\in A$, 
\[\frac{\Vol(X)}{\Vol(\sphere^m)}= \frac{\dim_{\R[f_0, \ldots, f_m]}A }{\deg f_0 \cdots \deg f_m}\]
where $\dim_{\R[f_0, \ldots, f_m]}A$ denotes the rank of $A$ as a finite free module over its subalgebra $\R[f_0, \ldots, f_m]$. 
\end{enumerate}
\end{theorem}
\begin{proof}

By Proposition \ref{P:CM}, the algebra $A$ is graded, finitely generated, and Cohen--Macaulay.

By Lemma \ref{L:sphericalWeyl}, the coefficients of the power series 
\[(1+z)H(z)=\sum_{k=0}^\infty b_k z^k\]
satisfy 
\[ b_k \sim C k^{m}\qquad \text{as } k\to \infty \]
where 
\[C={vol(X)\omega_{m}\over (2\pi)^{m}},\]
 $m=\dim X$ and $\omega_{m}=vol \DD^{m}$.

By Lemma \ref{L:asymptotics}, the Krull dimension of $A$ is $m+1$, the leading term in the Laurent series of $H(z)$ at $z=1$ is $\frac{Cm!}{2}(1-z)^{-m-1}$, and 
\[C=\frac{2\dim_{\R[f_0,\ldots, f_m]} A}  {m! \deg f_0 \cdots \deg f_m}.\]
Thus
\[\frac{\vol(X)}{\vol(\sphere^m)}  = \frac{\dim_{\R[f_0,\ldots, f_m]} A}  { \deg f_0 \cdots \deg f_m} \cdot 
\frac{2(2\pi)^{m} }{m! \omega_{m}\vol(\sphere^m) } \]

Finally, it remains to show that $\vol(\sphere^m)\omega_m= 2(2\pi)^m /m!$. This can be done with a straightforward induction argument using the well-known recursive formulas $\omega_{i+1}=\vol(\sphere^i)/(i+1)$ and $\vol(\sphere^{i+1})=2\pi \omega_i$.
\end{proof}

The analogous result to Theorem \ref{MT:spherical} for orthogonal representations is:
\begin{proposition}
\label{P:spherical}
Let $V$ be the Euclidean vector space of dimension $n+1$, and $G\to \O(V)$ an orthogonal representation of the compact group $G$. Let $A=\R[V]^G\subset\R[V]$ be the associated algebra of invariant polynomials, and let $H(z)$ denote its Hilbert series. Denote by $\sphere V$ the unit sphere in $V$, and by $m$ the dimension of the orbit space $X=\sphere V /G$.

Then:
\begin{enumerate}[(a)]
\item The Laurent series of $H(z)$ at $z=1$ has the form
\[H(z)= \frac{\Vol(X)}{\Vol(\sphere^m)} (1-z)^{-m-1} + (\cdots) (1-z)^{-m} + \cdots\]
\item $A$ is a Cohen--Macaulay algebra with Krull dimension $m+1$, and, for every homogeneous system of parameters $f_0, \ldots f_m\in A$, 
\[\frac{\Vol(X)}{\Vol(\sphere^m)}= \frac{\dim_{\R[f_0, \ldots, f_m]}A }{\deg f_0 \cdots \deg f_m}\]
where $\dim_{\R[f_0, \ldots, f_m]}A$ denotes the rank of $A$ as a finite free module over its subalgebra $\R[f_0, \ldots, f_m]$. 
\end{enumerate}
\end{proposition}
\begin{proof}
The proof is completely analogous to the proof of Theorem \ref{T:spherical}, except that the analogue of Lemma \ref{L:sphericalWeyl} for representations is proved using the main results of \cite{Donnelly78, BH78} instead of Theorem \ref{MT:Weyl}.
\end{proof}
Even though Proposition \ref{P:spherical} follows from \cite{Donnelly78, BH78} using classical results, to the best of the authors' knowledge this fact has not been noticed for general (infinite) groups $G$.

\subsection{Examples}
In this section we provide examples to illustrate Theorem \ref{MT:spherical} and its version for representations, Proposition \ref{P:spherical}.

\begin{example}\label{EX:finite}
Suppose $G$ is a \emph{finite} group, and $G\to\O(V)$ is injective (i.e., faithful). Then the orbit space $X=\sphere V/G$ is obtained from a fundamental domain of the action by boundary identifications. In particular $\dim X=\dim \sphere V$ and $\Vol(X)/\Vol(\sphere V)=|G|^{-1}$. Moreover, Proposition \ref{P:spherical}(a) follows immediately from Molien's formula, see  \cite[Proposition 3.1.4]{NeuselSmith}.
\end{example}

Specializing even more:
\begin{example}\label{EX:refl}
Suppose $G\subset \O(V)$ is a finite group generated by reflections. Then $A=\R[V]^G$ is free by the 
Chevalley--Shephard--Todd Theorem \cite[Theorem 7.1.4]{NeuselSmith}, so that Proposition \ref{P:spherical}(b) reduces to the well-known fact that 
\[|G|=\deg f_0 \cdots \deg f_m\] 
for an algebraically independent generating set $f_0, \ldots f_m$ of $\R[V]^G$, see \cite[Proposition 3.1.5]{NeuselSmith}
\end{example}

An example illustrating Theorem \ref{MT:spherical}:
\begin{example}[Clifford foliations]
The assertions and definitions below can be found in \cite{Radeschi14}.

Let $V=\R^{2l}$ and $C=(P_0, \ldots P_m)$ be a Clifford system on $V$, which implies $m\leq l$. Let $\psi\colon V\to \R^{m+2}$ be the polynomial map given by 
\[ \psi(x)=(\psi_1, \ldots, \psi_{m+2}) (x)= (\|x\|^2, \langle P_0 x, x\rangle, \ldots,  \langle P_{m} x, x\rangle)\]
If $m\neq l-1$, the fibers of $\psi$ form an infinitesimal singular Riemannian foliation $\F_C$ of $V$, called a ``Clifford foliation''. The leaf space $X=\sphere V/\F_C$ is isometric to a hemisphere in $\sphere^{m+1}(1/2)$ when $m\leq l-2$, and to $\sphere^m(1/2)$ when $m=l$. Thus the quotient of volumes $\vol X/\vol(\sphere^{\dim X})$ appearing in Theorem \ref{MT:spherical} is $2^{-m-2}$, respectively $2^{-m}$.

On the other hand,  $\psi_1, \ldots, \psi_{m+2}$ actually generate the algebra $A$ of basic polynomials of $\F_C$, see \cite[proof of Theorem C]{MR20q}. When $m\leq l-2$, they are algebraically independent because the image of $\psi$ has non-empty interior. Each $\psi_i$ has degree two, so this confirms the conclusion of Theorem \ref{MT:spherical}(b).

When $m=l$, there is one relation between the generators, namely $\psi_1^2= \psi_2^2 +\cdots +\psi_{m+2}^2$. The polynomials $\psi_1, \ldots, \psi_{m+1}$ form a homogeneous system of parameters, and $A$ is a free module of rank $2$ over $\R[\psi_1, \ldots, \psi_{m+1}]$. Thus the right-hand side of the equation in Theorem \ref{MT:spherical}(b) is $\frac{2}{2^{m+1}}=2^{-m}$, as expected.
\end{example} 

\begin{example}[Hopf fibrations]
Special cases of the previous example (all with m=l) are the infinitesimal singular Riemannian foliations of $\R^4=\C^2$, $\R^8=\Hr^2$, and $\R^{16}=\Ca^2$ whose restrictions to the unit sphere form the fibers of one of the Hopf fibrations $\sphere^3\to \sphere^2(1/2)$, $\sphere^7\to \sphere^4(1/2)$, and $\sphere^{15}\to \sphere^8(1/2)$.

Concretely in the complex case, the infinitesimal singular Riemannian foliation of $\R^4=\C^2$ is given by the orbits under the action of the unit complex numbers $a\in \sphere^1$ by $a(z,w)=(az,aw)$. The algebra $A$ of invariants is generated by the degree $2$ real polynomials $\psi_1, \ldots, \psi_4$ given by $|z|^2+|w|^2, |z|^2-|w|^2, 2\mathrm{Re}(z\bar{w}), 2\mathrm{Im}(z\bar{w})$, respectively. The only relation is $\psi_1^2= \psi_2^2 +\psi_3^2+ \psi_4^2$. A possible choice of system of parameters is $\psi_1, \psi_2, \psi_3$, for which  $\dim_{\R[\psi_1,\psi_2, \psi_3 ]}A =2$, thus illustrating part (b) of either Theorem \ref{MT:spherical} or Proposition \ref{P:spherical}, because $\frac{2}{2^3}=\vol \sphere^2(1/2) / \vol \sphere^2(1)$.\end{example}

\section{A note on manifold submetries}
\todo{Added a final blurb on Alexander's stuff. feel free to add/change/remove as you see fit}
An even more general concept than closed singular Riemannian foliations, is given by manifold submetries, which are maps $\sigma:M\to X$, from a smooth Riemannian manifold to a metric space, whose fibers are smooth and mutually equidistant. This generalizes closed singular Riemannian foliations in essentially two ways, first by not requiring the presence of smooth vector fields generating the tangent spaces of the fibers (although it is conjectured that such vector fields would exist anyway) and by allowing the fibers to be disconnected. The structure of submetries from Riemannian manifolds (and, in particular, manifold submetries) has been studied in \cite{KL20}. In the particular case where $M$ is the unit sphere in $\R^n$,  it was proved in \cite{MR20, MR23} that there is a 1-1 correspondence between manifold submetries from $M$ and Laplacian algebras contained in $\R[x_1, \ldots, x_n]$.

The statement of Theorem \ref{MT:Weyl} still makes sense for manifold submetries, and the authors believe it still holds. However, the main obstacle to generalizing our proof to the manifold submetry case is the lack of a Slice Theorem (cf. Section \ref{SS:slice}). More specifically, given $L=\sigma^{-1}(x)$, the differential $d_p\sigma: \sphere(\nu_p L)\to \Sigma_xX$ (where $\Sigma_xX$ is the space of directions of $X$ at $x$) might, in principle, depend on the choice of $p\in L$.

The Slice Theorem was used in a fundamental way in the inductive step to prove the geometric estimate \eqref{E:induction} in the proof of Theorem \ref{MT:geometricestimate}. In view of the discussion above, proving that the estimate \eqref{E:induction} still holds essentially requires a stronger version of Theorem \ref{MT:geometricestimate}, applied to a manifold submetry $\sigma:M\to X$, where the constant $C$ only depends on $X$, rather than the specific $\sigma$.

After this paper was finished, we were informed by A. Lytchak of his new manuscript \cite{Lytchak23}, in which he obtains a number of results bounding the geometry of leaves in a manifold submetry $\sigma:M\to X$, which only depend on the geometry of $M$ and $X$ rather than on $\sigma$ itself. In particular, he obtains a more general and accurate version of Lemma \ref{L:Jacobian} (cf. Proposition 1.3 and Corollary 7.4 of \cite{Lytchak23}).
The authors believe that the results in \cite{Lytchak23}, together with previous results on the structure of transnormal submetries in \cite{KL20}, could constitute a important ingredients to generalizing Theorem \ref{MT:geometricestimate} in the way mentioned above, and hence Theorem \ref{MT:Weyl}, to the manifold submetry case.



\bibliography{ref}

\begin{thebibliography}{GGR15}

\bibitem[ABT13]{ABT13}
Marcos~M. Alexandrino, Rafael Briquet, and Dirk T\"{o}ben.
\newblock Progress in the theory of singular {R}iemannian foliations.
\newblock {\em Differential Geom. Appl.}, 31(2):248--267, 2013.

\bibitem[AR15]{AR15}
Marcos~M. Alexandrino and Marco Radeschi.
\newblock Isometries between leaf spaces.
\newblock {\em Geom. Dedicata}, 174:193--201, 2015.

\bibitem[AR16]{AR16}
Marcos~M. Alexandrino and Marco Radeschi.
\newblock Mean curvature flow of singular {R}iemannian foliations.
\newblock {\em J. Geom. Anal.}, 26(3):2204--2220, 2016.

\bibitem[Bes08]{Besse}
Arthur~L. Besse.
\newblock {\em Einstein manifolds}.
\newblock Classics in Mathematics. Springer-Verlag, Berlin, 2008.
\newblock Reprint of the 1987 edition.

\bibitem[BH79]{BH78}
Jochen Br\"{u}ning and Ernst Heintze.
\newblock Representations of compact {L}ie groups and elliptic operators.
\newblock {\em Invent. Math.}, 50(2):169--203, 1978/79.

\bibitem[Bus10]{MR2742784}
Peter Buser.
\newblock {\em Geometry and spectra of compact {R}iemann surfaces}.
\newblock Modern Birkh\"{a}user Classics. Birkh\"{a}user Boston, Ltd., Boston,
  MA, 2010.
\newblock Reprint of the 1992 edition.

\bibitem[Cha84]{MR768584}
Isaac Chavel.
\newblock {\em Eigenvalues in {R}iemannian geometry}, volume 115 of {\em Pure
  and Applied Mathematics}.
\newblock Academic Press, Inc., Orlando, FL, 1984.
\newblock Including a chapter by Burton Randol, With an appendix by Jozef
  Dodziuk.

\bibitem[Cha06]{Chavel}
Isaac Chavel.
\newblock {\em Riemannian geometry}, volume~98 of {\em Cambridge Studies in
  Advanced Mathematics}.
\newblock Cambridge University Press, Cambridge, second edition, 2006.
\newblock A modern introduction.

\bibitem[DK15]{DerksenKemper}
Harm Derksen and Gregor Kemper.
\newblock {\em Computational invariant theory}, volume 130 of {\em
  Encyclopaedia of Mathematical Sciences}.
\newblock Springer, Heidelberg, enlarged edition, 2015.
\newblock With two appendices by Vladimir L. Popov, and an addendum by Norbert
  A'Campo and Popov, Invariant Theory and Algebraic Transformation Groups,
  VIII.

\bibitem[Don78]{Donnelly78}
Harold Donnelly.
\newblock {$G$}-spaces, the asymptotic splitting of {$L^{2}(M)$} into
  irreducibles.
\newblock {\em Math. Ann.}, 237(1):23--40, 1978.

\bibitem[Fel71]{Feller}
William Feller.
\newblock {\em An introduction to probability theory and its applications.
  {V}ol. {II}}.
\newblock John Wiley \& Sons, Inc., New York-London-Sydney, second edition,
  1971.

\bibitem[GGR15]{GGR15}
Fernando Galaz-Garcia and Marco Radeschi.
\newblock {Singular Riemannian foliations and applications to positive and
  non-negative curvature}.
\newblock {\em Journal of Topology}, 8(3):603--620, 04 2015.

\bibitem[GR15]{GR15}
Jianquan Ge and Marco Radeschi.
\newblock Differentiable classification of 4-manifolds with singular riemannian
  foliations.
\newblock {\em Mathematische Annalen}, 363(1):525--548, 2015.

\bibitem[H{\"o}r55]{BF02392492}
Lars H{\"o}rmander.
\newblock {On the theory of general partial differential operators}.
\newblock {\em Acta Mathematica}, 94(none):161 -- 248, 1955.

\bibitem[H{\"o}r68]{BF02391913}
Lars H{\"o}rmander.
\newblock {The spectral function of an elliptic operator}.
\newblock {\em Acta Mathematica}, 121(none):193 -- 218, 1968.

\bibitem[HT92]{HoweTan}
Roger Howe and Eng-Chye Tan.
\newblock {\em Nonabelian harmonic analysis}.
\newblock Universitext. Springer-Verlag, New York, 1992.
\newblock Applications of ${{\rm{S}}L}(2,{{\bf{R}}})$.

\bibitem[Isr]{MOanswer}
Robert Israel.
\newblock Asymptotics for the coefficients of a rational function.
\newblock MathOverflow.
\newblock URL:https://mathoverflow.net/q/103562 (version: 2012-08-01).

\bibitem[KL22]{KL20}
Vitali Kapovitch and Alexander Lytchak.
\newblock The structure of submetries.
\newblock {\em Geom. Topol.}, 26(6):2649--2711, 2022.

\bibitem[LR18]{LR18}
Alexander Lytchak and Marco Radeschi.
\newblock Algebraic nature of singular {R}iemannian foliations in spheres.
\newblock {\em J. Reine Angew. Math.}, 744:265--273, 2018.

\bibitem[Lyt23]{Lytchak23}
Alexander Lytchak.
\newblock Some regularity of submetries.
\newblock https://arxiv.org/abs/2307.14284, 2023.

\bibitem[Mol88]{Molino}
Pierre Molino.
\newblock {\em Riemannian foliations}, volume~73 of {\em Progress in
  Mathematics}.
\newblock Birkh\"{a}user Boston, Inc., Boston, MA, 1988.
\newblock Translated from the French by Grant Cairns, With appendices by
  Cairns, Y. Carri\`ere, \'{E}. Ghys, E. Salem and V. Sergiescu.

\bibitem[MP49]{MR31145}
S.~Minakshisundaram and \AA~. Pleijel.
\newblock Some properties of the eigenfunctions of the {L}aplace-operator on
  {R}iemannian manifolds.
\newblock {\em Canad. J. Math.}, 1:242--256, 1949.

\bibitem[MR19]{MR19}
Ricardo A.~E. Mendes and Marco Radeschi.
\newblock A slice theorem for singular {R}iemannian foliations, with
  applications.
\newblock {\em Trans. Amer. Math. Soc.}, 371(7):4931--4949, 2019.

\bibitem[MR20a]{MR20q}
R.~A.~E. Mendes and M.~Radeschi.
\newblock Singular {R}iemannian foliations and their quadratic basic
  polynomials.
\newblock {\em Transform. Groups}, 25(1):251--277, 2020.

\bibitem[MR20b]{MR20}
Ricardo A.~E. Mendes and Marco Radeschi.
\newblock Laplacian algebras, manifold submetries and the inverse invariant
  theory problem.
\newblock {\em Geom. Funct. Anal.}, 30(2):536--573, 2020.

\bibitem[MR23]{MR23}
Ricardo A.~E. Mendes and Marco Radeschi.
\newblock Maximality of {L}aplacian algebras, with applications to invariant
  theory.
\newblock {\em Ann. Mat. Pura Appl. (4)}, 202(2):1011--1031, 2023.

\bibitem[NS02]{NeuselSmith}
Mara~D. Neusel and Larry Smith.
\newblock {\em Invariant theory of finite groups}, volume~94 of {\em
  Mathematical Surveys and Monographs}.
\newblock American Mathematical Society, Providence, RI, 2002.

\bibitem[PR96]{PR96}
Efton Park and Ken Richardson.
\newblock The basic {L}aplacian of a {R}iemannian foliation.
\newblock {\em Amer. J. Math.}, 118(6):1249--1275, 1996.

\bibitem[Rad]{RadLN}
M.~Radeschi.
\newblock Singular riemannian foliations.
\newblock Lecture notes available at www.marcoradeschi.com.

\bibitem[Rad14]{Radeschi14}
Marco Radeschi.
\newblock Clifford algebras and new singular {R}iemannian foliations in
  spheres.
\newblock {\em Geom. Funct. Anal.}, 24(5):1660--1682, 2014.

\bibitem[Ric10]{Richardson10}
Ken Richardson.
\newblock Traces of heat operators on {R}iemannian foliations.
\newblock {\em Trans. Amer. Math. Soc.}, 362(5):2301--2337, 2010.

\bibitem[RS72]{ReedSimon}
Michael Reed and Barry Simon.
\newblock {\em Methods of Modern Mathematical Physics}, volume~1.
\newblock Academic Press, 1972.

\bibitem[Wei77]{Wei77}
Alan Weinstein.
\newblock Symplectic {$V$}-manifolds, periodic orbits of {H}amiltonian systems,
  and the volume of certain {R}iemannian manifolds.
\newblock {\em Comm. Pure Appl. Math.}, 30(2):265--271, 1977.

\bibitem[Wil06]{Wilf}
Herbert~S. Wilf.
\newblock {\em generatingfunctionology}.
\newblock A K Peters, Ltd., Wellesley, MA, third edition, 2006.

\end{thebibliography}
\bibliographystyle{alpha}

\end{document}